\documentclass[12pt,a4paper, titlepage]{amsart}
\usepackage{amsaddr}
\usepackage[latin1]{inputenc}

\usepackage{enumerate}

\usepackage{float}
\restylefloat{table}

\usepackage{graphicx}

\usepackage{hyperref}

\usepackage{amsmath}

\usepackage{amssymb}

\usepackage{amsthm}

\usepackage{tikz-cd}

\usepackage{a4wide}

\usepackage{bbm}

\newtheoremstyle{plain}
  {6pt}   % ABOVESPACE
  {6pt}   % BELOWSPACE
  {\itshape}  % BODYFONT
  {0pt}       % INDENT (empty value is the same as 0pt)
  {\bfseries} % HEADFONT
  {.}         % HEADPUNCT
  {5pt plus 1pt minus 1pt} % HEADSPACE
  {}          % CUSTOM-HEAD-SPEC

\newtheoremstyle{definition}
  {6pt}   % ABOVESPACE
  {6pt}   % BELOWSPACE
  {\normalfont}  % BODYFONT
  {0pt}       % INDENT (empty value is the same as 0pt)
  {\bfseries} % HEADFONT
  {.}         % HEADPUNCT
  {5pt plus 1pt minus 1pt} % HEADSPACE
  {}          % CUSTOM-HEAD-SPEC

\theoremstyle{plain}
\newtheorem*{thm*}{Theorem}
\newtheorem{thm}{Theorem}[section]
\newtheorem{prop}[thm]{Proposition}
\newtheorem{cor}[thm]{Corollary}
\newtheorem{lem}[thm]{Lemma}
\newtheorem{theorem}{Theorem}

\newtheorem{proposition}[theorem]{Proposition}

\theoremstyle{definition}
\newtheorem{defn}[thm]{Definition}

\newtheorem{ex}[thm]{Example}
\newtheorem{rmk}[thm]{Remark}

\numberwithin{equation}{thm}

\newcommand{\emphbf}[1]{\emph{\textbf{#1}}}

\DeclareMathAlphabet{\mathpzc}{OT1}{pzc}{m}{it}

\newcommand{\rad}[1]{\radoperator(#1)}

\DeclareMathOperator{\radoperator}{rad}
\DeclareMathOperator{\Kopf}{top}

\DeclareMathOperator{\id}{id}
\DeclareMathOperator{\Hom}{Hom}

\DeclareMathOperator{\Mod}{Mod}
\DeclareMathOperator{\modu}{mod}
\DeclareMathOperator{\rep}{rep}

\DeclareMathOperator{\coker}{Coker}
\DeclareMathOperator{\Rep}{Rep}

\DeclareMathOperator{\End}{End}

\DeclareMathOperator{\image}{Im}

\DeclareMathOperator{\Coker}{Coker}

\DeclareMathOperator{\add}{add}

\DeclareMathOperator{\injdim}{injdim}
\DeclareMathOperator{\projdim}{projdim}

\DeclareMathOperator{\Alg}{Alg}
\DeclareMathOperator{\Algpro}{Alg_{2pro}}
\DeclareMathOperator{\AlgMod}{AlgMod}
\DeclareMathOperator{\AlgproMod}{Alg_{2pro}Mod}
\DeclareMathOperator{\Algfree}{Alg_{2free}}
\DeclareMathOperator{\AlgfreeMod}{Alg_{2free}Mod}
\DeclareMathOperator{\rank}{rank}
\DeclareMathOperator{\arrowin}{in}
\DeclareMathOperator{\arrowout}{out}

\DeclareMathOperator{\sgn}{sgn}
\DeclareMathOperator{\mult}{mult}
\DeclareMathOperator{\sub}{sub}
\DeclareMathOperator{\fac}{fac}

\def\mymathhyphen{{\hbox{-}}}

\begin{document}

\title{Pro-species of algebras I: Basic properties}
\author{Julian K\"ulshammer}
\date{\today}

\address{
Institute of Algebra and Number Theory,
University of Stuttgart \\ Pfaffenwaldring 57 \\ 70569 Stuttgart,
Germany} \email{kuelsha@mathematik.uni-stuttgart.de}

\thanks{The author would like to thank Chrysostomos Psaroudakis, Sondre Kvamme, and the anonymous referee for helpful comments on previous versions of the paper.}


\begin{abstract}
In this paper, we generalise part of the theory of hereditary algebras to the context of pro-species of algebras. Here, a pro-species is a  generalisation of Gabriel's concept of species gluing algebras via projective bimodules along a quiver to obtain a new algebra. This provides a categorical perspective on a recent paper by Gei\ss, Leclerc, and Schr\"oer \cite{GLS16}. In particular, we construct a corresponding preprojective algebra, and establish a theory of a separated pro-species yielding a stable equivalence between certain functorially finite subcategories.
\end{abstract}

\maketitle

\section{Introduction}

The representation theory of finite dimensional hereditary algebras is among the best understood theories to date. Over algebraically closed fields, hereditary algebras are given by path algebras of finite acyclic quivers. Over more general fields, species, introduced by Gabriel \cite{Gab73}, form another class of hereditary algebras, which in the case that the ground field is perfect exhaust all finite dimensional hereditary algebras. Species can be regarded as skew fields glued via bimodules along a quiver to obtain an algebra. Their representation theory was studied intensively by Dlab and Ringel in a series of papers \cite{DR74a, DR74b, DR75, DR76, DR80}. 

Recently, Gei\ss, Leclerc, and Schr\"oer \cite{GLS16} defined algebras by quivers and relations, which turn out to have a species-like behaviour -- although they can as well be defined over algebraically closed fields. These can be viewed as various $k[x_i]/(x_i^{c_i})$ glued via bimodules which are free from both sides along a quiver. In \cite{GLS16} part of the representation theory of species has been generalised to these algebras resulting in an analogue of Gabriel's theorem. Their theory has been partially generalised to Frobenius algebras glued via  bimodules which are free from both sides by Fang Li and Chang Ye \cite{LY15}. They do not use the language of species but instead work with upper triangular matrix rings. In this paper, we generalise part of the theory of species to what we call pro-species of algebras, that is we generalise species by gluing arbitrary algebras (not necessarily skewfields) via bimodules which are projective from both sides along a quiver. 

Our goal is to give a conceptual approach to the papers \cite{GLS16} and \cite{LY15}, and provide some additional results for this theory. The philosophy is that the representation theory of a pro-species $\Lambda$ consisting of algebras $\Lambda_{\mathtt{i}}$ glued via bimodules $\Lambda_\alpha$ which are projective from both sides along a quiver $Q$ is in some parts governed by the individual representation theories for the $\Lambda_{\mathtt{i}}$. As a first example we restate results of Wang \cite{Wan16} and Luo and Zhang \cite{LZ13} on how to construct Iwanaga-Gorenstein algebras and describe their categories of Gorenstein projective modules as well as modules of finite projective dimension. 

A second part of the paper concerns the theory of reflection functors. In 1973, Bernstein, Gel'fand, and Ponomarev \cite{BGP73} introduced reflection functors for quivers in order to give a more conceptual proof of Gabriel's theorem characterising representation-finite path algebras of quivers. These functors were generalised to species by Dlab and Ringel in \cite{DR76}. In 1979, Gel'fand and Ponomarev \cite{GP79} introduced the preprojective algebra of the path algebra of an acyclic quiver by a certain doubling procedure of the original quiver. This algebra, regarded as a module over the original algebra, decomposes as the direct sum of all preprojective modules. Surprisingly recently, reflection functors have also been defined for preprojective algebras, independently by Baumann and Kramnitzer \cite{BK12}, Bolten \cite{Bol10}, and Buan, Iyama, Reiten, and Scott \cite{BIRS09}. The first two papers describe them in linear algebra terms similar to \cite{GP79} while \cite{BIRS09} gives a (in the non-Dynkin case) tilting module which provides this equivalence. See \cite{BKT14} for a comparison of the two approaches which was observed by Amiot. We extend this theory to the setting of pro-species.

\begin{theorem}[Section \ref{sec:reflectionfunctors}]\label{mainthm1}
Let $\Lambda$ be a pro-species of algebras. Then, there are reflection functors $(\Sigma_\mathtt{i}^+,\Sigma_\mathtt{i}^-)$ on the module category of the associated preprojective algebra $\Pi(\Lambda)$ of $\Lambda$. They can be described in terms of linear algebra as well as by $(\Hom_\Lambda(I_\mathtt{i},-), I_\mathtt{i}\otimes_\Lambda -)$ for a two-sided ideal $I_\mathtt{i}\subseteq \Pi(\Lambda)$.
\end{theorem}

A third section of this article generalises the theory of the separated quiver of a radical square zero algebra and the resulting stable equivalence, as obtained by Auslander and Reiten \cite{AR73, AR75}. 

\begin{theorem}[Theorem \ref{stableequivalence}]\label{mainthm2}
Let $\Lambda$ be a pro-species such that all the $\Lambda_\mathtt{i}$ are selfinjective. Let $\Gamma$ be the quotient of the tensor algebra $T(\Lambda)$  by its degree greater or equal to two part. Let $\Lambda^s$ be the separated pro-species of $\Lambda$. Then there is a stable equivalence of subcategories 
\[\underline{\modu}_{l.p.}\Gamma\to \underline{\rep}_{l.p.}\Lambda^s\]
where the subscript $l.p.$ denotes the subcategory of modules which are projective when regarded as modules for $\Lambda_\mathtt{i}$.
\end{theorem}

Furthermore, extending results of Ib\'a\~nez Cobos, Navarro, and L\'opez Pe\~na in the context of ``generalised path algebras'' we explain the quiver and relations given by Gei\ss, Leclerc, and Schr\"oer by the following

\begin{proposition}[Propositions \ref{quivertensoralgebra} and \ref{quiverpreprojectivealgebra}]\label{mainthm3}
Let $\Lambda$ be a pro-species of algebras such that algebras $\Lambda_\mathtt{i}$ associated to the vertices are given by path algebras of quivers with relations. Then, the tensor algebra $T(\Lambda)$ as well as the preprojective algebra $\Pi(\Lambda)$ have a description in terms of a quiver with relations using the  descriptions of the $\Lambda_\mathtt{i}$ in terms of quivers with relations and descriptions of the $\Lambda_\alpha$ as quotients of their respective projective cover as a bimodule for each arrow $\alpha$.
\end{proposition}

The article is structured as follows. In Section \ref{sec:species} we introduce the concept of a pro-species of algebras and the corresponding category of representations generalising the original notions due to Gabriel. Furthermore we define the associated tensor algebra $T(\Lambda)$ which has the same representation theory as $\Lambda$. 
Section \ref{sec:gorenstein} shows that $T(\Lambda)$ shows similar behaviour as a hereditary algebra and provides conditions under which it is Iwanaga-Gorenstein. 
In Section \ref{sec:preprojectivealgebra} we introduce the preprojective algebra $\Pi(\Lambda)$ associated to a pro-species of algebras $\Lambda$. Section \ref{sec:reflectionfunctors} contains the definition of the reflection functors and a proof of Theorem \ref{mainthm1}. Section \ref{sec:separated} describes the separated species for a pro-species of algebras and proves a stable equivalence between subcategories of locally projective representations, i.e. Theorem \ref{mainthm2}. The final section, Section \ref{sec:quivers}, which is independent of Sections \ref{sec:gorenstein}, \ref{sec:reflectionfunctors}, and \ref{sec:separated}, gives the quiver and relations of $T(\Lambda)$ and $\Pi(\Lambda)$ as stated in Proposition \ref{mainthm3}.

Throughout let $\mathbbm{k}$ be a field. Unless specified otherwise, modules are left modules. Unless stated otherwise, algebras and modules are assumed to be finite dimensional over $\mathbbm{k}$. For a quiver $Q$ we denote its set of vertices by $Q_0$, the set of arrows by $Q_1$, the set of paths in $Q$ including the length $0$ paths by $Q_p$, the set of paths excluding the length $0$ paths by $Q_p^+$, and by $s,t\colon Q_1\to Q_0$ the functions mapping an arrow to its starting, respectively terminating vertex. Furthermore, throughout we write $M\otimes_A g$ to mean $\id_M\otimes_A g$ for a right $A$-module $M$ and an $A$-linear map $g$.

\section{Pro-species of algebras}
\label{sec:species}

In this section, we generalise the notion of a species as defined by Gabriel \cite{Gab73}. The generalisation is similar to \cite{Li12} with the difference, that we do not start with a valued quiver which we want to ``modulate'' by algebras and bimodules. Instead we start with a gadget consisting of algebras and bimodules and only define (under certain conditions) the corresponding valued quiver. A further difference is the language of bicategories, which is not essential in the setting of this article, but the author hopes it will make it possible in the future to generalise some of the notions to other categories than free categories (i.e. path algebras of a quiver regarded as categories).

\begin{defn}
\begin{enumerate}[(i)]
\item The \emphbf{bicategory of algebras with bimodules}  $\Alg$ (or $\mathbbm{k}\mymathhyphen\Alg$ if we want to emphasise the ground field) is defined as follows:
\begin{description}
\item[objects] are $\mathbbm{k}$-algebras
\item[1-morphisms] the category of 1-morphisms between two objects $A$, $B$ is defined to be the class of finitely generated $B$-$A$-bimodules
\item[2-morphisms] bimodule homomorphisms
\item[1-composition] $[{}_CN_B]\circ [{}_BM_A]:=[{}_CN_B\otimes {}_BM_A]$
\item[2-composition] composition of bimodule homomorphisms
\item[identity] $\id_A={}_AA_A$
\end{description}
\item The \emphbf{bicategory of algebras with bimodules, projective from both sides} $\Algpro$ is the subcategory of  $\Alg$ with the same objects, but for morphisms only taking $B$-$A$-bimodules which are finitely generated projective from either side. Similarly we define $\Algfree$, the \emphbf{bicategory of algebras with bimodules which are free of finite rank from each side}.
\end{enumerate}
\end{defn}

It is easy to see that $\Algpro$ really is a sub-bicategory, i.e. that the tensor product of two bimodules which are projective from either side is again projective from either side and the identity is a module which is projective from either side.

\begin{defn}
Let $Q$ be a (finite) free $\mathbbm{k}$-category (i.e. the path algebra of a quiver, regarded as a category). 
\begin{enumerate}[(i)]
\item A \emphbf{pro-species of algebras} is a $\mathbbm{k}$-linear strict $2$-functor\footnote{It might seem unnatural to consider a strict $2$-functor (instead of a pseudofunctor) to a bicategory, but since $Q$ is a free $\mathbbm{k}$-category, this works. Many things in the sequel will carry over to the setting where $Q$ is not assumed to be free and $\Lambda$ is only assumed to be a pseudofunctor. On some occasions one has to replace a strict notion by the corresponding weak notion.}: $\Lambda\colon Q\to \Algpro$. We write $\Lambda_\mathtt{i}$ for $\Lambda(\mathtt{i})$ when $\mathtt{i}\in Q_0$ and $\Lambda_\alpha$ for $\Lambda(\alpha)$ when $\alpha\in Q_p$.
\item A pro-species of algebras is called a \emphbf{species of algebras} if $\Lambda\colon Q\to \Algfree$.
\end{enumerate}
\end{defn}

\begin{rmk}
Stated in more basic terms, a pro-species of algebras over $Q$ is a $\mathbbm{k}$-algebra $\Lambda_\mathtt{i}$ for each vertex $\mathtt{i}\in Q_0$ and a $\Lambda_\mathtt{j}$-$\Lambda_\mathtt{i}$-bimodule $\Lambda_\alpha$ for each arrow $\alpha\colon \mathtt{i}\to \mathtt{j}$.
\end{rmk}

\begin{ex}
\begin{enumerate}[(a)]
\item If $Q$ is the category with only one object and only scalar multiples of the identity, then a (pro)species of algebras is a $\mathbbm{k}$-algebra.
\item If $\Lambda_\mathtt{i}$ is a $\mathbbm{k}$-division ring for all $\mathtt{i}$, then a pro-species of algebras is a species in the sense of Gabriel, see \cite{Gab73}. A special case is when all $\Lambda_\mathtt{i}$ are in fact the ground field $\mathbbm{k}$, then such $\Lambda$ can be regarded as a $\mathbbm{k}$-quiver.
\end{enumerate}
\end{ex}

To a species of algebras one can associate a valued quiver which provides the link to \cite{Li12}.

\begin{defn} Let $Q$ be a quiver.
\begin{enumerate}[(i)]
\item A \emphbf{valuation} on $Q$ consists of two functions a function $c_0\colon Q_0\to \mathbb{Z}_{\geq 0}, \mathtt{i}\mapsto c_\mathtt{i}$ and $c_1\colon Q_1\to \mathbb{Z}^2_{\geq 0}, \alpha \mapsto (c_\alpha,c_{\alpha^*})$ such that $\forall \mathtt{i}\in Q_0$ there exists $c_\mathtt{i}>0$ with $c_{\alpha}c_{s(\alpha)}=c_{\alpha^*}c_{t(\alpha)}$.
\item Let $\Lambda$ be a species of algebras. Then, the \emphbf{associated valuation} is given by $c_\mathtt{i}:=\dim_k \Lambda_\mathtt{i}$ and $c_\alpha:=\rank_{\Lambda_{s(\alpha)}} \Lambda_\alpha$.
\end{enumerate}
\end{defn}

The next step is to introduce a notion of representation of a pro-species. For this we need another bicategory. This bicategory is in fact the lax coslice bicategory in $\Alg$ over $\mathbbm{k}$.
\footnote{The author would like to thank Pavel Safronov for this observation \href{http://mathoverflow.net/questions/201038}{http://mathoverflow.net/questions/201038}.}

\begin{defn}
\begin{enumerate}[(i)]
\item The \emphbf{bicategory of algebra-module pairs}  $\AlgMod$ (or $\mathbbm{k}\mymathhyphen\AlgMod$ if we want to emphasise the commutative ring we are working over) is defined as follows:
\begin{description}
\item[objects] are pairs $(A,N)$ where $A$ is a $\mathbbm{k}$-algebra and $N$ is an $A$-module,
\item[1-morphisms] the category of $1$-morphisms between two objects $(A,N)$, $(A',N')$ is defined to be the class of pairs $(M,g)$ where $M$ is an $A'$-$A$-bimodule and $g\colon M\otimes_A N\to N'$ is an $A'$-module homomorphism,
\item[2-morphisms] bimodule homomorphisms $\psi\colon (M,g)\to (\tilde{M},\tilde{g})$ such that $\tilde{g}\circ (\psi\otimes N)=g$,
\item[1-composition] $(M',g')\circ (M,g)=(M'\otimes_{A'} M, g'\circ (M'\otimes g))$,
\item[2-composition] composition of bimodule homomorphisms,
\item[identity] $\id_{(A,N)}=(A,\psi)$, where $\psi\colon A\otimes_A N\to N$ is the canonical identification.
\end{description}
\item The \emphbf{bicategory of algebra-module pairs which are projective from both sides} $\AlgproMod$ is the subcategory of $\AlgMod$ with the same objects, but for morphisms only taking $B$-$A$-bimodules which are projective from either side. Similarly we define $\AlgfreeMod$.
\item There is a forgetful functor $\Theta\colon \AlgMod\to \Alg$ forgetting about the module. It restricts to functors $\Theta\colon \AlgproMod\to \Algpro$ and $\Theta\colon \AlgfreeMod\to \Algfree$.
\end{enumerate}
\end{defn}

\begin{defn}
Let $\Lambda\colon Q\to \Algpro$ be a pro-species of algebras. 
\begin{enumerate}[(i)]
\item A \emphbf{representation of $\Lambda$} is a $\mathbbm{k}$-linear strict $2$-functor $M\colon Q\to \AlgproMod$ such that $\Theta M=\Lambda$. 
\item A \emphbf{morphism of $\Lambda$-representations $M\to N$} is a natural transformation $\alpha\colon M\Rightarrow N$ such that $\Theta(\alpha)=\id$. 
\end{enumerate}
All $\Lambda$-representations form a category $\Rep(\Lambda)$ with the composition of natural transformations and the identity natural transformation.
\end{defn}

\begin{rmk}
Again in more basic terms, a $\Lambda$-representation is a $\Lambda_\mathtt{i}$-representation $M_\mathtt{i}$ for each $\mathtt{i}\in Q_0$ and a $\Lambda_{t(\alpha)}$-linear map $M_\alpha\colon \Lambda_\alpha\otimes_{\Lambda_{s(\alpha)}} M_{s(\alpha)}\to M_{t(\alpha)}$ for each $\alpha\in Q_1$.
\end{rmk}

As pro-species of algebras can be regarded as a generalisation of the notion of a $\mathbbm{k}$-quiver, as usual, one can associate an algebra whose category of modules is equivalent to the category of representations of the pro-species.

\begin{defn}
Let $\Lambda\colon Q\to \mathbbm{k}\mymathhyphen\Algpro$ be a pro-species of algebras. Then the \emphbf{tensor algebra} $T(\Lambda)$ of $\Lambda$ is defined as follows. As a $\mathbbm{k}$-vector space it is:
\[T(\Lambda):=\prod_{\mathtt{i}\in Q_0} \Lambda_\mathtt{i}\oplus \bigoplus_{\alpha\in Q_p^+} \Lambda_\alpha\]
The multiplication is given as follows:
\begin{itemize}
\item $\prod \Lambda_\mathtt{i}$ has the usual multiplication of a product of algebras. 
\item By definition each $\Lambda_\alpha$ is an $\Lambda_{t(\alpha)}$-$\Lambda_{s(\alpha)}$-bimodule. Thus, it is also an $\prod \Lambda_\mathtt{i}$-$\prod \Lambda_\mathtt{i}$-bimodule via the projection maps. 
\item For $\lambda_\alpha\in \Lambda_\alpha$ and $\lambda_\beta\in \Lambda_\beta$ we define \[\lambda_\alpha\cdot \lambda_\beta:=\begin{cases}\lambda_\alpha\otimes \lambda_\beta\in \Lambda_{\alpha\beta}=\Lambda_\alpha\otimes\Lambda_\beta&\text{if } t(\beta)=s(\alpha)\\0&\text{else.}\end{cases}\]
\end{itemize}
By $\varepsilon_\mathtt{i}$ we denote the identity element of $\Lambda_\mathtt{i}$, considered as an element of $T(\Lambda)$. 
\end{defn}

Note that $T(\Lambda)$ is unital if and only if $Q$ is finite. In this case $1_{T(\Lambda)}=\sum_{\mathtt{i}\in Q_0} \varepsilon_\mathtt{i}$ is a decomposition into orthogonal idempotents (which is primitive if and only if all the $\Lambda_\mathtt{i}$ are local).  Also note that $T(\Lambda)$ is finite dimensional if and only if $Q$ has no oriented cycles. Unless stated otherwise, we will assume these two properties of $T(\Lambda)$ from now on. 

\begin{ex}\label{examples}
\begin{enumerate}[(a)]
\item Let $A$ be a $\mathbbm{k}$-algebra. If $\Lambda\colon Q\to \mathbbm{k}\mymathhyphen\Algpro$ is defined by $\Lambda_\mathtt{i}=A$ and $\Lambda_\alpha=A$ for all $\mathtt{i}\in Q_0$ and $\alpha\in Q_1$, then $T(\Lambda)=A\otimes_\mathbbm{k} \mathbbm{k}Q$ is the path algebra of $Q$ over $A$.
\item Let $A,B$ be $\mathbbm{k}$-algebras and $M$ an $A$-$B$-bimodule. Let $Q=(\mathtt{1}\stackrel{\alpha}{\to} \mathtt{2})$ be the quiver of Dynkin type $\mathbb{A}_2$. Then, for $\Lambda$ defined by $\Lambda_\mathtt{1}=B$, $\Lambda_\mathtt{2}=A$, $\Lambda_\alpha=M$ we have that $T(\Lambda)\cong\begin{pmatrix}A&M\\0&B\end{pmatrix}$, with multiplication being the usual matrix multiplication. 
\item An extension of case (b) is considered in \cite{LY15} where the authors study $n\times n$-upper triangular matrices (with Frobenius algebras on the diagonal). We explain the relationship: For $\mathtt{i}=\mathtt{1},\dots, \mathtt{n}$ let $A_{\mathtt{i}}$ be algebras and for $\mathtt{i}<\mathtt{j}$ let $B_{\mathtt{i}\mathtt{j}}$ be $A_{\mathtt{i}}$-$A_{\mathtt{j}}$-bimodules which are projective from either side. Define a quiver $Q$ with vertices $\mathtt{1},\dots,\mathtt{n}$ and exactly one arrow $\mathtt{j}\to \mathtt{i}$ whenever $B_{\mathtt{i}\mathtt{j}}\neq 0$. Let $\Lambda\colon Q\to \Algpro$ be the pro-species of algebras defined by $\Lambda_{\mathtt{i}}=A_{\mathtt{i}}$ and $\Lambda_{\alpha}:=B_{\mathtt{i}\mathtt{j}}$ for the unique arrow $\alpha\colon \mathtt{j}\to \mathtt{i}$. Let 
\[A_{\mathtt{i}\mathtt{j}}:=\bigoplus_{l=0}^{\mathtt{j-i-1}}\bigoplus_{\mathtt{i}<\mathtt{k}_1<\mathtt{k}_2<\dots<\mathtt{k}_l<\mathtt{j}}B_{\mathtt{i}\mathtt{k}_1}\otimes_{A_{\mathtt{k}_1}} B_{\mathtt{k}_1\mathtt{k}_2}\otimes_{A_{\mathtt{k}_2}}\dots\otimes_{A_{\mathtt{k}_l}} B_{\mathtt{k}_l\mathtt{j}}.\]
Then, $T(\Lambda)\cong \begin{pmatrix}A_\mathtt{1}&A_{\mathtt{12}}&\dots&A_{\mathtt{1n}}\\0&A_2&\dots&A_{\mathtt{2n}}\\\vdots&\vdots&\ddots&\vdots\\0&0&\dots&A_{\mathtt{n}}\end{pmatrix}$.

Conversely, let $Q$ be an acyclic quiver with $\mathtt{n}$ vertices and $\Lambda\colon Q\to \Algpro$ be a pro-species of algebras. Number the vertices of $Q$ such that all the arrows go in decreasing direction of arrows. Define $B_{\mathtt{ij}}:=\bigoplus_{\alpha\colon \mathtt{j}\to \mathtt{i}} \Lambda_\alpha$ (which should be taken as $0$ if there are no arrows $\mathtt{j}\to \mathtt{i}$). Define $A_{\mathtt{ij}}$ as above. Then $T(\Lambda)$ is isomorphic to the triangular matrix ring defined as above.
\end{enumerate}
\end{ex}

\begin{prop}
Let $\Lambda\colon Q\to \Algpro$ be a pro-species of algebras. The categories of representations of $\Lambda$ and modules over $T(\Lambda)$ are equivalent. In particular, $\operatorname{Rep}(\Lambda)$ is an abelian category.
\end{prop}

\begin{proof}
Let $\Theta'\colon \AlgMod\to \Mod \mathbbm{k}$ be the projection on the second component. One easily checks  that $\Phi\colon \operatorname{Rep}(\Lambda)\to \Mod(T(\Lambda))$ defined via $M\mapsto \bigoplus \Theta' M(x)$ with the $\prod \Lambda_\mathtt{i}$-action given by the projections $\prod \Lambda_\mathtt{i}\to \Lambda_\mathtt{j}$ and the action of $\Lambda_\alpha$ given by $\Theta'(M(\alpha))$, is an equivalence with inverse functor given by $V\mapsto M$ with $M(\mathtt{i})=(\Lambda_\mathtt{i},\varepsilon_\mathtt{i}V)$ for all $\mathtt{i}\in Q_0$ and $M(\alpha)$ is given by the restriction of the action $T(\Lambda)\otimes V\to V$ to $\Lambda_\alpha\otimes \varepsilon_\mathtt{i}V\to \varepsilon_\mathtt{j}V$ for an arrow $\alpha\colon \mathtt{i}\to \mathtt{j}$ .
\end{proof}

In some of the later sections it will turn out that the representation theory of a pro-species $\Lambda$ is in some way glued together from the individual representation of the algebras $\Lambda_\mathtt{i}$ sitting on the vertices. This observation yields to the following notions.

\begin{defn}\label{locally}
\begin{enumerate}[(i)]
\item Let $\Lambda\colon Q\to \Alg$ be a pro-species of algebras. We say that a property \emphbf{holds locally} if it holds for all $\Lambda_\mathtt{i}$.
\item Let $M\colon Q\to \AlgproMod$ be a $\Lambda$-representation. We say that a property \emphbf{holds locally} if it holds for all $\Lambda_\mathtt{i}$-modules $M_\mathtt{i}$.
\end{enumerate} 
\end{defn}

The general philosophy should be that if a property holds locally for a pro-species of algebras $\Lambda$, then (a slightly weaker version of) this property holds for the algebra $T(\Lambda)$.\\
Furthermore it should be of interest to understand the category of all $\Lambda$-representations for which a certain local property holds. For example, the results of \cite{GLS16, LY15} show that under certain conditions the representations which are locally free  behave like the (ordinary) representations of $Q$.

\section{Iwanaga-Gorenstein algebras}
\label{sec:gorenstein}

In this section, we provide cases in which $T(\Lambda)$ behaves like a hereditary algebra. Furthermore, we give instances of the general philosophy claimed at the end of the foregoing section. Namely, the algebras $\Lambda_\mathtt{i}$ all being Iwanaga-Gorenstein results in $T(\Lambda)$ being Iwanaga-Gorenstein. In these cases, the category of Gorenstein projective modules as well as the category of modules of finite global dimension have been described locally.

The following is a generalisation of \cite[Proposition 3.1]{GLS16} (see also {\cite[Proposition 2.6]{LY15}}).

\begin{prop}\label{tensoralgebralocallyprojective}
Let $\Lambda\colon Q\to \Algpro$ be a pro-species of algebras. Then, the following statements hold. 
\begin{enumerate}[(i)]
\item\label{tensor:i} The algebra $T(\Lambda)$ regarded as a $\Lambda$-representation is locally projective.
\item\label{tensor:ii} There is the following short exact sequence of $T(\Lambda)$-modules:
\[0\to \bigoplus_{\substack{p\in Q_p^+\\s(\alpha)=\mathtt{i}}}\Lambda_\alpha \to T(\Lambda)\varepsilon_\mathtt{i}\to \Lambda_\mathtt{i}\to 0.\]
In particular, $\projdim_{T(\Lambda)}\Lambda_\mathtt{i}\leq 1$ for all $\mathtt{i}\in Q_0$.
\end{enumerate}
\end{prop}

\begin{proof}
For \eqref{tensor:i} note that $\varepsilon_\mathtt{i}T(\Lambda)=\Lambda_\mathtt{i}\oplus \bigoplus_{t(p)=\mathtt{i}}\Lambda_p$. Hence $\varepsilon_\mathtt{i}T(\Lambda)$ is a projective $\Lambda_\mathtt{i}$-module. 

For \eqref{tensor:ii} note that if we set $\deg \Lambda_p=|p|$, the length of the path $p$ in $Q$ and $\deg \Lambda_\mathtt{i}=0$ then $T(\Lambda)$ is a graded algebra. With respect to this grading, the projection $T(\Lambda)\varepsilon_\mathtt{i}\to \Lambda_\mathtt{i}$ is the projection on the degree $0$ component and is therefore $T(\Lambda)$-linear. Its kernel is obviously $\bigoplus_{s(p)=\mathtt{i}} \Lambda_p$. It remains to prove that the kernel is a projective $T(\Lambda)$-module. This follows from the fact that $\bigoplus_{s(p)=\mathtt{i}} \Lambda_p=\bigoplus_{s(p)=\mathtt{i}}\bigoplus_{p=q\beta, \beta\in Q_1(\mathtt{i},\mathtt{j}), q\in Q_p}\Lambda_q\otimes_{\Lambda_\mathtt{j}} \Lambda_\beta$, and the fact that since $\Lambda_q$ is a projective $\Lambda_{t(q)}$-module, this is a direct summand of $\bigoplus_{\beta\colon \mathtt{i}\to \mathtt{j}}(T(\Lambda)\varepsilon_\mathtt{j})^{m_\mathtt{j}}$ for some $m_\mathtt{j}$.
\end{proof}

The following lemma describing a bimodule resolution of $T(\Lambda)$ generalises \cite[Proposition 7.1]{GLS16} and \cite[Lemma 3.3]{LY15}. All statements are special cases of \cite[Theorems 10.1 and 10.5]{Sch85}.

\begin{prop}
There is a short exact sequence of $T(\Lambda)$-$T(\Lambda)$-bimodules
\[
\begin{tikzcd}
P_\bullet\colon 0\arrow{r}&\bigoplus_{\alpha\in Q_1} T(\Lambda)\varepsilon_{t(\alpha)}\otimes_{\Lambda_{t(\alpha)}} \Lambda_\alpha\otimes_{\Lambda_{s(\alpha)}}\varepsilon_{s(\alpha)}T(\Lambda)\\
{\phantom{P_\bullet\colon 0}}\arrow{r}{d}&\bigoplus_{\mathtt{i}\in Q_0} T(\Lambda)\varepsilon_\mathtt{i}\otimes_{\Lambda_\mathtt{i}}\varepsilon_\mathtt{i} T(\Lambda)\arrow{r}{\mult}&T(\Lambda)\arrow{r} &0
\end{tikzcd}
\]
where $\mult$ denotes the natural multiplication and $d(p\otimes h\otimes q):=ph\otimes q-p\otimes hq$.
\end{prop}

As an application we obtain projective resolutions not only of $\Lambda_\mathtt{i}$, but of all locally projective $\Lambda$-modules. We denote the category of such modules by $\rep_{l.p.}(\Lambda)$. 
It is easy to see that (for $Q$ acyclic) this category is the extension closure of $\add\{\Lambda_\mathtt{i}\, |\, \mathtt{i}\in Q_0\}$. Further, let $\rep_{l.p.}^{np}(\Lambda)$ be the full subcategory of $\Lambda$-representations without projective direct summands. 

\begin{cor}
Let $\Lambda$ be a pro-species of algebras. Let $M\in \rep_{l.p.}(\Lambda)$. Then $P_\bullet\otimes_{T(\Lambda)} M$ is a projective resolution of $M$ which explicitly looks as follows
\[
\begin{tikzcd}
0\arrow{r}&\bigoplus_{\alpha\in Q_0} T(\Lambda)\varepsilon_{t(\alpha)}\otimes_{\Lambda_{t(\alpha)}}\Lambda_\alpha\otimes_{\Lambda_{s(\alpha)}} M_{s(\alpha)}\\
{}\arrow{r}{d\otimes M}&\bigoplus_{\mathtt{i}\in Q_0} T(\Lambda)\varepsilon_\mathtt{i}\otimes_{\Lambda_\mathtt{i}} M_\mathtt{i}\arrow{r}{\mult} &M\arrow{r}&0
\end{tikzcd}
\]
with $(d\otimes M)(p\otimes h\otimes m)=ph\otimes m-p\otimes M_\alpha(h\otimes m)$.
\end{cor}

\begin{proof}
As $T(\Lambda)$ is projective as a right $T(\Lambda)$-module, the short exact sequence $P_\bullet$ splits as right $T(\Lambda)$-modules. Thus, the sequence remains exact by tensoring with $M$ and can be identified with the claimed sequence. If $M$ is locally projective, then the two left terms are indeed projective $T(\Lambda)$-modules by a similar argument as in the proof of Proposition \ref{tensoralgebralocallyprojective}. The claim follows.
\end{proof}

We recall the definition of an Iwanaga-Gorenstein algebra, sometimes also called Gorenstein algebra. This is a generalisation of the class of selfinjective algebras (which is the case $n=0$). For an introduction to the topic, see e.g. \cite{Che10}.

\begin{defn}
An algebra $A$ is called \emphbf{$n$-Iwanaga-Gorenstein} if $\injdim_A A\leq n$ and $\projdim_AD(A)\leq n$. 
\end{defn}

The following proposition generalises \cite[Theorem 1.1]{GLS16} where the case $n=0$ is proven for a specific example (see also \cite[Corollary 2.8]{LY15}). Using the language of triangular matrix rings, it is proven in \cite[Corollary 3.6, Proposition 3.8]{Wan16}. For a slightly different proof, it is possible to generalise the proof of \cite[Theorem 1.1]{GLS16}.

\begin{prop}\label{locallyGorenstein}
Let $\Lambda$ be a locally $n$-Iwanaga-Gorenstein pro-species in the sense of Definition \ref{locally}. Then the following are equivalent for a $T(\Lambda)$-module $U$:
\begin{enumerate}[(1)]
\item\label{locallyGorenstein:1} $\projdim U\leq n+1$,
\item\label{locallyGorenstein:2} $\projdim U<\infty$,
\item\label{locallyGorenstein:3} $\injdim U\leq n+1$,
\item\label{locallyGorenstein:4} $\injdim U<\infty$,
\item\label{locallyGorenstein:5} $U$ locally has projective dimension $\leq n$,
\item\label{locallyGorenstein:6} $U$ locally has injective dimension $\leq n$.
\end{enumerate}
In particular $T(\Lambda)$ is $(n+1)$-Gorenstein.
\end{prop}

The following corollary was stated as \cite[Theorem 3.9]{GLS16} for the special case $n=0$. There, it was proved directly generalising the methods in \cite{AR91}.

\begin{cor}
Let $\Lambda$ be a locally $n$-Iwanaga-Gorenstein pro-species. Then the category of modules which are locally of projective dimension $\leq n$ is functorially finite, resolving and coresolving and has Auslander--Reiten sequences.
\end{cor}

\begin{proof}
This follows from the previous proposition, since for a general $(n+1)$-Iwanaga-Gorenstein algebra, the category of modules of finite projective dimension has the stated properties, see \cite[Corollary 2.3.6]{Che10}.
\end{proof}

Another important subclass of the category of modules for a Gorenstein algebra is the category of Gorenstein projective modules. 

\begin{defn}
Let $A$ be an algebra and $M$ be an $A$-module. A \emphbf{projective biresolution} 
\[P^\bullet\colon \cdots \to P^{n-1}\stackrel{d^{n-1}}{\to} P^n\stackrel{d^n}{\to} P^{n+1}\to \cdots\]
of $M$ is an exact sequence of projectives with $\ker d^0\cong M$.

A projective biresolution is called \emphbf{complete} if $\Hom_A(P^\bullet,A)$ is exact, or equivalently $\Hom_A(P^\bullet,Q)$ is projective for every projective $A$-module $Q$.

An $A$-module $M$ is called \emphbf{Gorenstein projective} if there exists a complete projective biresolution of $M$.
\end{defn}

The following result generalises \cite[Theorem 4.1]{LZ13} from the case $n=2$ and \cite[Theorem 2.4]{LY15b} from the context of generalised path algebras (under stronger  assumptions), see also \cite[Corollary 1.5]{XZ12} for the case that $n=2$ and $T(\Lambda)$ is Gorenstein. 

\begin{prop}
Let $\Lambda\colon Q\to \Algpro$ be a pro-species. Then, $X\in T(\Lambda)$ is Gorenstein projective if and only if $M_{\mathtt{i},\arrowin}$ is injective and $\Coker(X_{\mathtt{i},\arrowin})$ is Gorenstein projective over $\Lambda_\mathtt{i}$ for every $\mathtt{i}\in Q_0$.
\end{prop} 

\begin{proof}
The proof follows the strategy given in \cite{LZ13} for $n=2$. Let $P_{\mathtt{i}}^{\bullet}$ be the complete projective biresolution with $\ker d^0_\mathtt{i}=\Coker(X_{\mathtt{i},\arrowin})$. We inductively define projective biresolutions 
\[Q_{\mathtt{i}}^\bullet=\bigoplus_{\substack{p\in Q_p\\t(p)=\mathtt{i}}}\Lambda_p\otimes_{\Lambda_{s(p)}} P_{s(p)}^\bullet\] of $X_{\mathtt{i}}$.
 For the start of the induction let $\mathtt{i}$ be a source in the quiver. Then, $X_\mathtt{i}=\Coker X_{\mathtt{i},\arrowin}$ and the statement holds by assumption. Next assume that $\mathtt{i}$ is a vertex such that $X_{\mathtt{j}}$ already was defined such a projective biresolution for all $\mathtt{j}$ such that there is a path $\mathtt{j}\to \mathtt{i}$. Since $X_{\mathtt{i},\arrowin}$ is a monomorphism, there is an exact sequence
\[0\to \bigoplus_{\substack{\alpha\in Q_1\\t(\alpha)=\mathtt{i}}} \Lambda_\alpha\otimes_{\Lambda_{s(\alpha)}} X_{s(\alpha)}\to X_\mathtt{i}\to \Coker(X_{\mathtt{i},\arrowin})\to 0\]
Using that $X_{s(\alpha)}$ and $\Coker(X_{\mathtt{i},\arrowin})$ have biresolutions $Q_{s(\alpha)}^\bullet$ and $P_{\mathtt{i}}^\bullet$, respectively, using a version of the horseshoe lemma and the fact that projectives are injective in the category of Gorenstein projective modules, claimed biresolution $Q_{\mathtt{i}}^\bullet$ also exists for $\mathtt{i}$, cf. \cite[Proof of Corollary 1.5]{LZ13}. Furthermore, such resolution can be chosen to have upper triangular differential with the differentials induced by those of the $P_{s(p)}^\bullet$ on the diagonal. 

Putting these biresolutions $Q_\mathtt{i}^\bullet$ together one obtains a biresolution $Q^\bullet$ of $X$. Here, for $\alpha\colon \mathtt{i}\to \mathtt{j}$, the corresponding map $\Lambda_\alpha\otimes_{\Lambda_\mathtt{i}} Q_\mathtt{i}^\bullet\to Q_\mathtt{j}^\bullet$ is just the direct sum of the identities on $\Lambda_\alpha\otimes_{\Lambda_{\mathtt{i}}} \Lambda_p\otimes P_{s(p)}^{\bullet}$. 

We need to check that $\Hom_\Lambda(Q^\bullet,T(\Lambda)\varepsilon_\mathtt{i})$ is again exact for every $\mathtt{i}\in Q_0$. Considering the diagram
\[
\begin{tikzcd}
\Lambda_\alpha \otimes \Lambda_p\otimes P_{s(p)}\arrow{r}{\Lambda_\alpha\otimes \varphi_{p,p'}}\arrow{d}{\id} &\Lambda_\alpha\otimes \Lambda_{p'}\otimes \Lambda_{s(p')}\arrow{d}\\
\Lambda_{\alpha p}\otimes P_{s(p)}\arrow{r}{\varphi_{\alpha p, q'}} &\Lambda_{q'} \otimes \Lambda_{s(q')}
\end{tikzcd}
\]
where $\varphi\in \Hom_\Lambda(Q^\bullet, T(\Lambda)\varepsilon_\mathtt{i})$ we see that $\varphi_{\alpha p,q'}=\begin{cases}\Lambda_\alpha\otimes \varphi_{p,p'}&\text{if $q'=\alpha p'$}\\0&\text{else.}\end{cases}$

Thus 
\[\Hom_\Lambda(Q^\bullet,T(\Lambda)\varepsilon_{\mathtt{i}})=\bigoplus_j \bigoplus_{\substack{p\in Q_p\\p\colon \mathtt{i}\to \mathtt{j}}} \Hom_{\Lambda_{\mathtt{j}}}(P^\bullet_{\mathtt{j}},\Lambda_p\otimes \Lambda_{\mathtt{i}})\]
as $\mathbb{Z}$-graded vector spaces with upper triangular differential (coming from the upper triangular differential of $Q^\bullet_{\mathtt{i}}$). As the $P_{\mathtt{j}}^\bullet$ are complete projective resolutions and $\Lambda_p\otimes \Lambda_{\mathtt{i}}$ is projective, it follows  $\Hom_{\Lambda_{\mathtt{j}}}(P_{\mathtt{j}},\Lambda_p\otimes  \Lambda_{\mathtt{i}})$ is exact for every $\mathtt{i}$. Using induction on $|Q_0|$, applying homology it follows that also $\Hom_{\Lambda}(Q^\bullet, T(\Lambda)\varepsilon_{\mathtt{i}})$ is exact for every $\mathtt{i}\in Q_0$. Thus, $Q^\bullet$ is a complete projective biresolution. The claim follows.

For the other direction, let $Q^\bullet$ be a complete projective biresolution of $X$. By the form of the projective $T(\Lambda)$-modules, this gives projective biresolutions $Q_{\mathtt{i}}^\bullet$ of $X_\mathtt{i}$ and furthermore $Q_\mathtt{i}^\bullet\cong \bigoplus_{\substack{p\in Q_p\\t(p)=\mathtt{i}}} \Lambda_p\otimes P_{s(p)}^\bullet$. 
Consider the differential $d$ on the complex $Q$. From the fact that $d$ is a $T(\Lambda)$-module homomorphism we get that the following diagram commutes (for every arrow $\alpha$, every $j\in \mathbb{Z}$ and every paths $p,p',q'$ in the quiver with $t(p)=t(p')$ and $t(\alpha)=t(q')$ :
\[
\begin{tikzcd}
\Lambda_\alpha\otimes \Lambda_p\otimes P_{s(p)}^j\arrow{d}{\id}\arrow{r}{\Lambda_\alpha\otimes d_{p,p'}^j} &\Lambda_\alpha \otimes \Lambda_{p'}\otimes P_{s(p')}^{j-1}\arrow{d}\\
\Lambda_{\alpha p}\otimes P_{s(p)}^{j}\arrow{r}{d_{\alpha p,q'}^j} &\Lambda_{q'}\otimes P_{s(q')}^{j-1}
\end{tikzcd}
\]
where the right vertical arrow is $\id$ or $0$ when $q'=\alpha p'$ or $q'\neq \alpha p'$ respectively. Thus, the differential $d$ can be assumed to be upper triangular and by induction $P^\bullet_{\mathtt{i}}$ is a complex for all $\mathtt{i}$. 
 
From applying homology to the sequence,
\[0\to \bigoplus_{\substack{p\in Q_p^+\\t(p)=\mathtt{i}}} \Lambda_p\otimes P_{s(p)}^\bullet\to Q_{\mathtt{i}}^\bullet\to P_{\mathtt{i}}^\bullet\to 0\]
it follows by induction that each of the complexes $P_\mathtt{i}^\bullet$ is in fact exact.

The claim follows by applying the snake lemma to the following diagram:
\[
\begin{tikzcd}
0\arrow{r}&\bigoplus_{\substack{p\in Q_p^+\\t(p)=\mathtt{i}}}\Lambda_p\otimes X_{s(p)}\arrow{r}\arrow{d} &\bigoplus_{\substack{p\in Q_p^+\\t(p)=\mathtt{i}}} \Lambda_p\otimes P_{s(\alpha)}^0\arrow{d}\arrow{r}&\dots\\
0\arrow{r}&X_\mathtt{i}\arrow{r}&P_\mathtt{i}\oplus \bigoplus_{\substack{p\in Q_p^+\\t(p)=\mathtt{i}}} \Lambda_p\otimes P_{s(\alpha)}^0\arrow{r}&\dots
\end{tikzcd}
\] 
as $\Coker(X_{\mathtt{i},\arrowin})=\Coker(\bigoplus_{\substack{p\in Q_p\\t(p)=\mathtt{i}}} \Lambda_p\otimes X_{s(p)}\to X_{\mathtt{i}})$ as all other components factor through some $X_\alpha$ for $\alpha\in Q_1$. 
\end{proof}

The following lemma shows that under appropriate conditions $T(\Lambda)$ deserves to be called ``hereditary''. 

\begin{lem}\label{tensoralgebrahereditary}
Let $\Lambda\colon Q\to \Algpro$ be a locally selfinjective pro-species with $Q$ not necessarily acyclic. Then, the following properties hold:
\begin{enumerate}[(i)]
\item[(i)] Every locally projective submodule of a projective module is projective.
\item Let $\rep_{l.p.}^{np}(\Lambda)$ be the subcategory of representations of $\Lambda$ without projective direct summands. Then the natural functor $\rep_{l.p.}^{np}(\Lambda)\to \underline{\rep}_{l.p.}(\Lambda)$ is an equivalence of categories.
\end{enumerate}
\end{lem}

\begin{proof}
\begin{enumerate}[(i)]
\item Let $P$ be projective and let $E\subseteq P$ be a submodule. Since $E_\mathtt{i}\subseteq P_{\mathtt{i}}$ and $\Lambda$ is locally selfinjective, the inclusion splits and $E$ and $P/E$ are locally projective. Consider the short exact sequence $0\to E\to P\to P/E\to 0$. As $P/E$ is locally projective, it follows from Proposition \ref{locallyGorenstein} that $\projdim P/E\leq 1$. Thus, \cite[Proposition A.4.7]{ASS06} implies that $E$ is projective.
\item Let $g\colon M\to N$ factor through a projective object $P$, i.e. $g=ts$ with $s\colon M\to P$ and $t\colon P\to N$. By the first part, it follows that $\image s\subseteq P$ is projective. It follows that $\image s$ is a direct summand of $M$. Thus, $\image s=0$ and hence, $g=0$. The claim follows.\qedhere
\end{enumerate}
\end{proof}

\section{The preprojective algebra}
\label{sec:preprojectivealgebra}

In this section, we introduce the notion of the preprojective algebra of a pro-species generalising \cite{DR80, GLS16} and establish some of its basic properties. For this level of generality we need 
the classical existence result of a dual basis of a projective module, cf. \cite[Lemma 2.4]{LY15}. The following lemma is well known, it can e.g. be found in lecture notes on Advanced Algebra by Pareigis. Since these are not published, we include a proof for the convenience of the reader.

\begin{lem}[Dual basis lemma]\label{dualbasislemma}
Let $P_R$ be a right $R$-module. Then, the following statements are equivalent:
\begin{enumerate}[(1)]
\item\label{dualbasis:i} $P$ is finitely generated and projective.
\item\label{dualbasis:ii} There are $x_1,\dots, x_m\in P$, $f_1,\dots,f_m\in \Hom_{R}(P,R)$ such that for each $x\in P$ we have $x=\sum_{i=1}^m  x_if_i(x)$
\item\label{dualbasis:iii} The dual basis homomorphism $\mathbbm{d}\colon P\otimes_R \Hom_R(P,R)\to \End_{R}(P), p\otimes f\mapsto (q\mapsto pf(q))$ is an isomorphism.
\end{enumerate}
\end{lem}

\begin{proof}
The equivalence of \eqref{dualbasis:i} and \eqref{dualbasis:ii} is proved e.g. in \cite[I.3.23]{Fai81}. For \eqref{dualbasis:ii}$\Rightarrow$\eqref{dualbasis:iii} let $\mathbbm{b}\colon \End_R(P)\to P\otimes_R \Hom_R(P,R)$ be defined by $\mathbbm{b}(e):=e(x_i)\otimes_R f_i$. We claim that $\mathbbm{b}$ is inverse to $\mathbbm{d}$. Indeed, $\mathbbm{d}\mathbbm{b}(e)(p)=\sum_i\mathbbm{d}(e(x_i)\otimes_R f_i)(p)=\sum_{i}e(x_i)f_i(p)=e(\sum_i x_if_i(p))=e(p)$ and $\mathbbm{b}\mathbbm{d}(p\otimes f)=\sum_i \mathbbm{d}(p\otimes f)(x_i)\otimes_R f_i=\sum_i pf(x_i)\otimes_R f_i=p\otimes_R \sum_i f(x_i)f_i=p\otimes f$ since $\sum_i f(x_i)f_i(q)=f(x_if_i(q))=f(q)$.

On the other hand, \eqref{dualbasis:iii}$\Rightarrow$\eqref{dualbasis:ii} follows since writing $\mathbbm{d}^{-1}(\id_P)=\sum_i x_i\otimes f_i$ the elements $x_i, f_i$ satisfy the stated condition as $\sum_i x_if_i(x)=\mathbbm{d}(x_i\otimes f_i)(x)=\mathbbm{d}\mathbbm{d}^{-1}(\id_P)(x)=x$.
\end{proof} 

\begin{cor}\label{uniquenesscasimirelement}
Let $P_R$ be a finitely generated projective right $R$-module. The element $\sum_{i=1}^m x_i\otimes f_i\in P\otimes_R \Hom_R(P,R)$ does not depend on the choice of $x_i, f_i$ satisfying \eqref{dualbasis:ii} in Lemma \ref{dualbasislemma}.
\end{cor}

\begin{proof}
Any elements $x_i, f_i$ satisfying \eqref{dualbasis:ii} in Lemma \ref{dualbasislemma} give $\mathbbm{d}\mathbbm{b}(\sum_{i=1}^m f_i\otimes x_i)=\id_P$. Hence, by \eqref{dualbasis:iii} in Lemma \ref{dualbasislemma} they have to coincide.
\end{proof}

\begin{defn}
Let $P$ be a finitely generated projective left $R$-module and $x_i, f_i$ as in \eqref{dualbasis:ii} of Lemma \ref{dualbasislemma}. Then, the element $c:=\sum_{i=1}^m  x_i\otimes f_i\in  P\otimes \Hom_R(P,R)$ is called the \emphbf{Casimir element}.
\end{defn}

Although not strictly necessary for defining the preprojective algebra of a pro-species, the following additional property is needed for it to enjoy some of the usual properties, e.g. of the preprojective algebra being independent of the orientation. It appeared already in \cite[Section 5.1]{GLS16} and \cite[Definition 1.1]{LY15} without being given a name.

\begin{defn}
A pro-species of algebras $\Lambda\colon Q\to \Algpro$ is called \emphbf{dualisable} if 
\[\Hom_{\Lambda_\mathtt{i}^{op}}(\Lambda_\alpha,\Lambda_\mathtt{i})\cong \Hom_{\Lambda_\mathtt{j}}(\Lambda_{\alpha},\Lambda_\mathtt{j})\]
 as $\Lambda_\mathtt{i}$-$\Lambda_\mathtt{j}$-bimodules for each $\alpha\colon \mathtt{i}\to \mathtt{j}\in Q_1$.
\end{defn}

\begin{defn}\label{preprojectivealgebra}
Let $\Lambda\colon Q\to \Algpro$ be a dualisable pro-species of algebras. 
\begin{enumerate}[(i)]
\item The \emphbf{double quiver} $\overline{Q}$ of $Q$  is the quiver with $\overline{Q}_0=Q_0$ and $\overline{Q}_1=Q_1\cup Q_1^*$ where $Q_1^*:=\{\alpha^*\colon \mathtt{j}\to \mathtt{i}| \alpha\colon \mathtt{i}\to \mathtt{j} \in Q_1\}$.
\item Let $\overline{\Lambda}\colon \overline{Q}\to \Algpro$ be the pro-species of algebras with $\overline{\Lambda}_\mathtt{i}:=\Lambda_\mathtt{i}$ for $\mathtt{i}\in Q_0$ and $\overline{\Lambda}_\alpha=\Lambda_\alpha$ for $\alpha\in Q_1$ and $\overline{\Lambda}_{\alpha^*}:=\Hom_{\Lambda_{s(\alpha)}^{op}}(\Lambda_\alpha, \Lambda_{s(\alpha)})$. 
\item For each $\alpha\in \overline{Q}_1$ let $c_\alpha$ be the Casimir element of $\Lambda_\alpha\otimes \Lambda_{\alpha^*}$. 
Define $c=\sum_{\alpha\in \overline{Q}_1}\sgn(\alpha)c_\alpha$, where $\sgn(\alpha)=\begin{cases}1&\text{if $\alpha\in Q_1$}\\-1 &\text{else}\end{cases}$. 
The algebra $\Pi(\Lambda):=T(\overline{\Lambda})/\langle c\rangle$ is called ``the'' \emphbf{preprojective algebra} of $\mathcal{M}$.
\end{enumerate}
\end{defn} 

\begin{ex}
Let $Q=\mathbb{A}_2$ and $\Lambda\colon Q\to \Algpro$ be a dualisable pro-species of algebras. Then, $\Pi(\Lambda)$ is isomorphic to the Morita ring $\Lambda_{(0,0)}$ as studied by Green and Psaroudakis in \cite{GP14}. 
\end{ex}

\begin{rmk}Note that for some of the purposes the signs $\operatorname{sgn}(\alpha)$ could be omitted. In particular, if the underlying unoriented graph is a tree, then $\Pi(\Lambda)$ does not depend on the signs up to isomorphism. In applications it seems that these signs are more ``natural'' than other choices, see \cite[Section 6]{Ri98}.
\end{rmk}

The following proposition generalises \cite[Lemma 1.1]{DR80}

\begin{lem}
Let $\Lambda$ be a dualisable pro-species of algebras. Then, $\Pi(\Lambda)$ does not depend on the choice of the dual bases of the $\Lambda_\alpha$ and $\Lambda_{\alpha^*}$. 
\end{lem}

\begin{proof}
This follows immediately from Corollary \ref{uniquenesscasimirelement}.
\end{proof}

As noted before, the preprojective algebra $\Pi(\Lambda)$ does not depend on the orientation of $Q$ in the following sense.

\begin{lem}
Let $\Lambda\colon Q\to \Algpro$ be a dualisable pro-species of algebras.  Choose any orientation $Q'$ of $Q$. Define a new pro-species $\Lambda'\colon Q'\to \Algpro$ by $\Lambda_\mathtt{i}'=\Lambda_\mathtt{i}$ for $\mathtt{i}\in Q'_0=Q_0$ and  
\[\Lambda_\alpha' =\begin{cases}\Lambda_\alpha&\text{if $\alpha\in Q_1$}\\\Hom(\Lambda_\alpha,\Lambda_{s(\alpha)})&\text{if $\alpha=\beta^*$ for $\beta\in Q_1$}\end{cases}\]
Then $\Lambda'$ is dualisable and $\Pi(\Lambda')\cong \Pi(\Lambda)$.
\end{lem}

\begin{proof}
Since $\Lambda$ is dualisable and projective modules are reflexive, it is immediate that $\Hom_{\Lambda_{t(\alpha)}}(\Lambda_{\alpha^*},\Lambda_{t(\alpha)})\cong \Lambda_\alpha$. Furthermore, if $\Lambda_\alpha$ is identified with its double dual then $f_1,\dots, f_m\in \Lambda_{\alpha^*}, x_1,\dots,x_m\in \Lambda_\alpha$ give a dual basis of $\Lambda_{\alpha^*}\otimes \Lambda_\alpha$ (as noted in the proof of Lemma \ref{dualbasislemma} $f=\sum_i f(x_i)f_i$). As in the classical case, interchanging the roles of $\alpha$ and $\alpha^*$, changing the signs of one of them, then yields an isomorphism between $\Pi(\Lambda)$ and $\Pi(\Lambda')$.
\end{proof}

In the remainder of this section, we introduce some prerequisites for the definition of reflection functors in the following section. 

\begin{prop}[cf. {\cite[Exercise 2.20]{Lam99}}]
Let $R$ and $S$ be rings. Let $M$, $N$ be $R$-modules. Then, there is a linear map $f\colon \Hom_R(M,R)\otimes_R N\to \Hom_R(M,N)$ sending $\varphi\otimes n$ to the map $m\mapsto \varphi(m)n$. If $M$ is also an $S^{op}$-module, then this map is a homomorphism of $S$-modules. If $M$ is finitely generated and $R$-projective, then $f$ is bijective with inverse given by $\psi\mapsto \sum_{i}f_i\otimes \psi(x_i)$ where $x_i,f_i$ are chosen as in Lemma \ref{dualbasislemma}.
\end{prop}

\begin{proof}
Define $f\colon \Hom_R(M,R)\times N\to \Hom_R(M,N)$ by $(\varphi,n)\mapsto (m\mapsto \varphi(m)n)$. This is $R$-balanced as $(\varphi r)(m):=\varphi(m)r$. If $M$ is also an $S^{op}$-module, then $\Hom_R(M,R)$ and $\Hom_R(M,N)$ have an $S$-module structure defined by $(s\varphi)(m)=\varphi(ms)$. It is obvious that $f$ is $S$-linear. A direct calculation shows that the claimed map is indeed inverse.
\end{proof}

\begin{defn}
Let $\Lambda$ be a dualisable pro-species of algebras, $M_{s(\alpha)}$ a left $\Lambda_{s(\alpha)}$-module and $N_{t(\alpha)}$ a left $\Lambda_{t(\alpha)}$ module. Then, combining the tensor-hom adjunction with the previous proposition we obtain an isomorphism 
\begin{align*}
\Hom_{\Lambda_{t(\alpha)}}(\Lambda_\alpha\otimes M_{s(\alpha)},N_{t(\alpha)})&\to \Hom_{\Lambda_{s(\alpha)}}(M_{s(\alpha)},\Hom_{\Lambda_{t(\alpha)}}(\Lambda_{\alpha},N_{t(\alpha)}))\\
&\to \Hom_{\Lambda_{s(\alpha)}}(M_{s(\alpha)}, \Lambda_{\alpha^*}\otimes N_{t(\alpha)}).
\end{align*}
We denote it by $f\mapsto f^\vee$, its inverse by $f\mapsto f^\wedge$.
\end{defn}

This isomorphism is used in Section \ref{sec:reflectionfunctors} to construct reflection functors in the context of pro-species of algebras.

Let $\Lambda$ be a dualisable pro-species of algebras. Let $M\in \modu T(\overline{\Lambda})$. Recall that each arrow $\alpha\in \overline{Q}_1$ defines a map $M_\alpha\colon \Lambda_\alpha\otimes_{\Lambda_{s(\alpha)}} M_{s(\alpha)}\to M_{t(\alpha)}$ and a map $M_\alpha^\vee\colon M_{s(\alpha)}\to \Lambda_{\alpha^*}\otimes_{\Lambda_{t(\alpha)}} M_{t(\alpha)}$. Define 
\[M_{\mathtt{i},\arrowin}:=\bigoplus_{\substack{\alpha\in \overline{Q}_1\\t(\alpha)=\mathtt{i}}}\sgn(\alpha)M_{\alpha} \text{\quad and $M_{\mathtt{i},\arrowout}:=\bigoplus_{\substack{\alpha\in \overline{Q}_1\\s(\alpha)=\mathtt{i}}}M_\alpha^{\vee}$}.\]
The following proposition generalises \cite[Proposition 5.2]{GLS16}.

\begin{prop}\label{inout}
Let $\Lambda$ be a dualisable pro-species of algebras. The category $\modu \Pi(\Lambda)$ is equivalent to the full subcategory of $\modu T(\overline{\Lambda})$ whose modules satisfy $M_{\mathtt{i},\arrowin}\circ M_{\mathtt{i},\arrowout}=0$ for all $\mathtt{i}\in Q_0$.
\end{prop}

\begin{proof}
For an object $M=(M_\mathtt{i},M_\alpha,M_{\alpha^*})_{\mathtt{i},\alpha}$ in $\modu T(\overline{\Lambda})$ note that 
\[M_{\mathtt{i},\arrowin}\circ M_{\mathtt{i},\arrowout}=\sum_{\substack{\alpha\in \overline{Q}\\s(\alpha)=\mathtt{i}}} \sgn(\alpha) M_{\alpha^*}\circ M_\alpha^\vee.\]
Recall that $M_\alpha^\vee$ is given explicitly by $M_\alpha^\vee(m)=\sum_{j}f_j\otimes M_\alpha(x_j\otimes m)$ for all $m\in M_{s(\alpha)}$. Hence,
\[M_{\mathtt{i},\arrowin}\circ M_{\mathtt{i},\arrowout}(m)=\sum_{\substack{\alpha\in \overline{Q}\\s(\alpha)=\mathtt{i}}} \sgn(\alpha)\sum_j M_{\alpha^*}(f_j \otimes M_\alpha(x_j\otimes m)),\]
which vanishes if and only if $M\in \Pi(\Lambda)$. 
\end{proof}

\section{BGP reflection functors}
\label{sec:reflectionfunctors}

In this section, we generalise BGP reflection functors to our setup. We start by giving a partial description of a projective bimodule resolution of the preprojective algebra analogous to \cite[Lemma 8.1.1]{GLS07}.

\begin{lem}\label{bimoduleresolutionpreprojective}
Let $\Lambda\colon Q\to \Algpro$ be a dualisable pro-species of algebras. Then, the following is the start of a projective bimodule resolution of $\Pi(\Lambda)$.
\[
\begin{tikzcd}
Q_\bullet\colon \bigoplus_{\mathtt{i}\in Q_0} \Pi(\Lambda)\otimes_{\Lambda_\mathtt{i}}\Pi(\Lambda) \arrow{r}{d^1}&\bigoplus_{\alpha\in \overline{Q}_1}\Pi(\Lambda)\otimes_{\Lambda_{t(\alpha)}} \Lambda_\alpha\otimes_{\Lambda_{s(\alpha)}} \Pi(\Lambda)\\
{\phantom{Q_\bullet\colon \bigoplus_{\mathtt{i}\in Q_0} \Pi(\Lambda)\otimes_{\Lambda_\mathtt{i}}\Pi(\Lambda)}}\arrow{r}{d^0}&\bigoplus_{\mathtt{i}\in Q_0}\Pi(\Lambda)\otimes_{\Lambda_\mathtt{i}} \Pi(\Lambda)
\end{tikzcd}
\]
where $d^0(1\otimes x\otimes 1)=x\otimes 1-1\otimes x$ for $x\in \Lambda_\alpha$ and $d^1(1\otimes 1)=\sum_{\substack{\alpha\in \overline{Q}_1\\s(\alpha)=\mathtt{i}}}\sgn(\alpha)(c_\alpha\otimes 1+1\otimes c_\alpha)$.
\end{lem}

\begin{proof}
It is obvious that $\Coker(d^0)=\Pi(\Lambda)$. Noting that $Q_\bullet$ is the total complex of a double complex, the statement follows from the acyclic assembly lemma \cite[(2.7.3)]{Wei94} in a similar way to \cite[Theorem 3.15]{BBK02}.
\end{proof}

For $\mathtt{i}\in Q_0$ define $I_\mathtt{i}:=\Pi(\Lambda)(1-\varepsilon_\mathtt{i})\Pi(\Lambda)$ to be the annihilator of the $\Pi(\Lambda)$-module $\Lambda_\mathtt{i}$. Define functors $\Sigma^+_\mathtt{i}:=\Hom_{\Pi(\Lambda)}(I_\mathtt{i},-)\colon \modu \Pi(\Lambda)\to \modu \Pi(\Lambda)$ and $\Sigma^-_\mathtt{i}:=I_i\otimes_{\Pi(\Lambda)} -\colon \modu \Pi(\Lambda)\to \modu\Pi(\Lambda)$. It follows that $(\Sigma_\mathtt{i}^+,\Sigma_\mathtt{i}^-)$ is a pair of adjoint functors, $\Sigma_\mathtt{i}^+$ being left exact and $\Sigma_\mathtt{i}^-$ being right exact. We will give an explicit description of the functors $\Sigma_\mathtt{i}^+$ and $\Sigma_\mathtt{i}^-$ generalising results by Baumann and Kamnitzer in \cite{BK12, BKT14}, and independently by Bolten in \cite{Bol10}. This in turn also generalises results by Geiss, Leclerc, and Schr\"oer in \cite{GLS16} for the case that $\Lambda_\mathtt{i}=k[x_\mathtt{i}]/(x_\mathtt{i}^{c_\mathtt{i}})$ and $\Lambda_\alpha$ a bimodule which is free from either side.  

Let $M\in \modu \Pi(\Lambda)$. Recall from Proposition \ref{inout} that $M_{\mathtt{i},\arrowout}$ induces a map $M_\mathtt{i}\to \ker M_{\mathtt{i},\arrowin}$ which will be denoted by $\overline{M}_{\mathtt{i},\arrowout}$. Define $N_\mathtt{i}:=\ker(M_{\mathtt{i},\arrowin})$ and the exact sequence 
\[
\begin{tikzcd}0\arrow{r} &N_\mathtt{i}\arrow{r}{(N_\alpha)^T_\alpha} &\bigoplus_{\substack{\alpha\in \overline{Q}_1\\t(\alpha)=\mathtt{i}}}\Lambda_\alpha\otimes M_{s(\alpha)}\arrow{r}{(M_\alpha)_{\alpha}} &M_\mathtt{i}\arrow{r} &0
\end{tikzcd}
\]
Let $\Sigma^+_\mathtt{i}(M)\in \modu \Pi(\Lambda)$ be defined by
\begin{align*}
\Sigma^+_\mathtt{i}(M)_\mathtt{j}&:=\begin{cases}M_\mathtt{j}&\text{if $\mathtt{j}\neq \mathtt{i}$,}\\N_\mathtt{i}&\text{if $\mathtt{j}=\mathtt{i}$,}\end{cases}\\
\Sigma^+_\mathtt{i}(M)_\alpha&:=\begin{cases}M_\alpha &\text{if $t(\alpha)\neq \mathtt{i}$ \text{ and } $s(\alpha)\neq \mathtt{i}$,}\\
N_{\alpha^*}^\wedge&\text{if $s(\alpha)=\mathtt{i}$,}\\ (\overline{M}_{\mathtt{i},\arrowout})_\alpha\circ M_\alpha &\text{if $t(\alpha)=\mathtt{i}$.}\end{cases}
\end{align*}

Dually if we denote $N_\mathtt{i}:=\Coker(M_{\mathtt{i},\arrowout})$ and by $\overline{M}_{\mathtt{i},\arrowin}$ the induced from $M_{\mathtt{i},\arrowin}$ map $\Coker(M_{\mathtt{i},\arrowout})\to M_\mathtt{i}$ and the exact sequence
\[
\begin{tikzcd}
0\arrow{r} &M_\mathtt{i}\arrow{r}{(M_\alpha^\vee)^T_\alpha} &\bigoplus_{\substack{\alpha\in \overline{Q}_1\\s(\alpha)=\mathtt{j}}}\Lambda_{\alpha^*}\otimes M_{t(\alpha)}\arrow{r}{(N_{\alpha})_{\alpha}} &N_\mathtt{i}\arrow{r} &0),
\end{tikzcd}
\]
$\Sigma_\mathtt{i}^-(M)\in \modu \Pi(\Lambda)$ can be defined via:
\begin{align*}
\Sigma^-_\mathtt{i}(M)_\mathtt{j}&:=\begin{cases}M_\mathtt{j}&\text{if $\mathtt{j}\neq \mathtt{i}$,}\\N_\mathtt{i}&\text{if $\mathtt{j}=\mathtt{i}$,}\end{cases}\\
\Sigma^-_\mathtt{i}(M)_\alpha&:=\begin{cases}M_\alpha &\text{if $s(\alpha)\neq \mathtt{i}$ and $t(\alpha)\neq \mathtt{i}$,}\\
N_{\alpha^*}&\text{if $t(\alpha)=\mathtt{i}$,}\\
(M_\alpha^{\vee}\circ (\overline{M}_{\mathtt{i},\arrowin})_{\alpha^*})^\wedge &\text{if $s(\alpha)=\mathtt{i}$}.\end{cases} 
\end{align*}

The constructions $\Sigma_\mathtt{i}^{\pm}$ can be extended to functors by the universal property of the kernel and the cokernel. The following lemma generalises \cite[Proposition 5.1]{BKT14}

\begin{lem}
The two definitions of $\Sigma_\mathtt{i}^+$ and $\Sigma_\mathtt{i}^-$ coincide.
\end{lem}

\begin{proof}
We only show the statement for $\Sigma_\mathtt{i}^-$, the proof for $\Sigma_\mathtt{i}^+$ is similar, taking into account the remarks from \cite[Proof of Proposition 5.1]{BKT14}. Applying $\Lambda_\mathtt{i}\otimes_{\Pi(\Lambda)}-$ to the resolution given in Lemma \ref{bimoduleresolutionpreprojective} (changing the sign of $\Lambda_\mathtt{i}\otimes d^0$) and replacing the rightmost terms by the kernel of $\Lambda_\mathtt{i}\otimes_{\Pi(\Lambda)} \mult$,  we obtain an exact sequence (again since the former sequence is split as right $\Pi(\Lambda)$-modules) of $\Lambda_\mathtt{i}$-$\Pi(\Lambda)$-bimodules:
\[
\begin{tikzcd}
\Lambda_\mathtt{i}\otimes_{\Lambda_\mathtt{i}} \Pi(\Lambda)\arrow{r}{\partial_1}&\bigoplus_{\substack{\alpha\in \overline{Q}_1\\t(\alpha)=\mathtt{i}}}\Lambda_\alpha\otimes_{\Lambda_{s(\alpha)}} \Pi(\Lambda)\arrow{r}{\partial_0}&\varepsilon_\mathtt{i} I_\mathtt{i}\arrow{r}&0
\end{tikzcd}
\]
where $\partial_1(1\otimes 1)=\sum_{\substack{\alpha\in \overline{Q}_1\\t(\alpha)=\mathtt{i}}} c_\alpha$ and $\partial_0(x\otimes 1)=x$. Let $M\in \modu \Pi(\Lambda)$. Applying $-\otimes_{\Pi(\Lambda)} M$ to this sequence, the resulting exact sequence of left $\Lambda_\mathtt{i}$-modules can be identified with
\[
\begin{tikzcd}
M_\mathtt{i}\arrow{r}{M_{\mathtt{i},\arrowout}} &\bigoplus_{\alpha\in \overline{Q}_1} \Lambda_\alpha\otimes_{\Lambda_{s(\alpha)}} M\arrow{r}&\varepsilon_\mathtt{i} I_\mathtt{i}\otimes_{\Pi(\Lambda)} M\arrow{r}&0
\end{tikzcd}
\]
As $I_\mathtt{i}=(1-\varepsilon_\mathtt{i})\Pi(\Lambda)\oplus \varepsilon_\mathtt{i} I_\mathtt{i}$, $I_\mathtt{i}\otimes_{\Pi(\Lambda)} M$ only changes $M$ at the $\mathtt{i}$-th component, the foregoing sequence tells us that as $\Pi(\Lambda)$-modules, the two definitions coincide.
\end{proof}

For $M\in \modu \Pi(\Lambda)$, let $\sub_\mathtt{i}(M)$ be the largest submodule $U$ of $M$ such that $\varepsilon_\mathtt{i} U=U$. Dually, let $\fac_\mathtt{i}(M)$ be the largest factor module $M/U$ of $M$ such that $\varepsilon_\mathtt{i} (M/U)=M/U$. These two constructions are functorial. The following generalises \cite[Proposition 9.1 (ii), Corollary 9.2]{GLS16}.
 
\begin{cor}
\begin{enumerate}[(i)]
\item There are functorial short exact sequences 
\[
0\to \sub_\mathtt{i}\to \id\to \Sigma_\mathtt{i}^+\Sigma_\mathtt{i}^-\to 0
\]
and 
\[0\to \Sigma_\mathtt{i}^-\Sigma_\mathtt{i}^+\to \id\to \fac_\mathtt{i}\to 0.
\]
\item The functors $\Sigma_\mathtt{i}^+$ and $\Sigma_\mathtt{i}^-$ restrict to inverse equivalences $\Sigma_\mathtt{i}^+\colon \mathcal{T}_\mathtt{i}\to \mathcal{S}_\mathtt{i}$ and $\Sigma_\mathtt{i}^-\colon \mathcal{S}_\mathtt{i}\to \mathcal{T}_\mathtt{i}$ where 
\[\mathcal{T}_\mathtt{i}:=\{M\in \modu \Pi(\Lambda) | \fac_\mathtt{i}(M)=0\}\] and 
\[\mathcal{S}_\mathtt{i}:=\{M\in \modu \Pi(\Lambda)| \sub_\mathtt{i}(M)=0\}.\]
\end{enumerate}
\end{cor}

\begin{proof}
The proof given for \cite[Proposition 9.1]{GLS16} applies verbatim.
\end{proof}

As in the classical case, we can restrict $\Sigma_\mathtt{i}^+$ (resp. $\Sigma_\mathtt{i}^-$) to $\modu T(\Lambda)$ provided $\mathtt{i}$ is a sink (respectively a source) in $Q$. Let $s_\mathtt{i}(Q)$ be the quiver with vertex set $Q_0$ and arrows $\{\alpha \in Q_1| t(\alpha)\neq \mathtt{i}\} \cup \{\beta^*\colon \mathtt{i}\to \mathtt{j} | (\beta\colon \mathtt{j}\to \mathtt{i})\in Q_1\}$ (respectively $\{\alpha\in Q_1| s(\alpha)\neq \mathtt{i}\}\cup \{\beta^*\colon \mathtt{j}\to \mathtt{i} | (\beta\colon \mathtt{i}\to \mathtt{j})\in Q_1\}$). Define the dualisable pro-species of algebras $s_\mathtt{i}(\Lambda)\colon s_\mathtt{i}(Q)\to \Algpro$ as follows:
\begin{align*}
(s_\mathtt{i}(\Lambda))_\mathtt{j}&:=\Lambda_\mathtt{j}&\text{ for all $\mathtt{j}\in Q_0$,}\\
(s_\mathtt{i}(\Lambda))_\alpha&:=\Lambda_\alpha&\text{if $t(\alpha)\neq \mathtt{i}$ (respectively $s(\alpha)\neq \mathtt{i}$),}\\
(s_\mathtt{i}(\Lambda))_{\beta^*}&:=\Hom_{\Lambda_\mathtt{i}}(\Lambda_\beta,\Lambda_\mathtt{i})&\text{if $t(\beta)=\mathtt{i}$ (respectively $t(\beta)= \mathtt{i}$).}
\end{align*}

Define the \emphbf{reflection functor} $F_\mathtt{i}^+\colon \rep(\Lambda)\to \rep(s_\mathtt{i}(\Lambda))$ (respectively $F_\mathtt{i}^-\colon \rep(\Lambda)\to \rep(s_\mathtt{i}(\Lambda))$) as follows. Let $M\in \rep(\Lambda)$. 
Define $N_\mathtt{i}:=\ker(M_{\mathtt{i},\arrowin})$ (respectively $N_\mathtt{i}:=\coker(M_{\mathtt{i},\arrowout})$). Then, there is an exact sequence 
\[
\begin{tikzcd}0\arrow{r} &N_\mathtt{i}\arrow{r}{(N_\alpha)^T_\alpha} &\bigoplus_{\substack{\alpha\in Q_1\\t(\alpha)=\mathtt{i}}}\Lambda_\alpha\otimes M_{s(\alpha)}\arrow{r}{(M_\alpha)_{\alpha}} &M_\mathtt{i}\arrow{r} &0
\end{tikzcd}
\]
(respectively an exact sequence 
\[
\begin{tikzcd}
0\arrow{r} &M_\mathtt{i}\arrow{r}{(M_\alpha^\vee)^T_\alpha} &\bigoplus_{\substack{\alpha\in Q_1\\s(\alpha)=\mathtt{i}}}\Lambda_{\alpha^*}\otimes M_{t(\alpha)}\arrow{r}{(N_{\alpha})_{\alpha}} &N_\mathtt{i}\arrow{r} &0).
\end{tikzcd}
\]
Then, define the representation $F_\mathtt{i}^+(M)$ by
\begin{align*}
F_\mathtt{i}^+(M)_j&:=\begin{cases}M_\mathtt{j}&\text{if $\mathtt{j}\neq \mathtt{i}$,}\\N_\mathtt{i}&\text{if $\mathtt{j}=\mathtt{i}$,}\end{cases}\\
F_\mathtt{i}^+(M)_\alpha&:=M_\alpha&\text{if $t(\alpha)\neq \mathtt{i}$,}\\
F_\mathtt{i}^+(M)_{\beta^*}&:=N_\beta^{\wedge}&\text{if $t(\beta)=\mathtt{i}$.}
\end{align*}
(respectively the representation $F_\mathtt{i}^-(M)$ by
\begin{align*}
F_\mathtt{i}^-(M)_\mathtt{j}&:=\begin{cases}M_\mathtt{j}&\text{if $\mathtt{j}\neq \mathtt{i}$,}\\N_\mathtt{i}&\text{if $\mathtt{j}=\mathtt{i}$,}\end{cases}\\
F_\mathtt{i}^-(M)_\alpha&:=M_\alpha&\text{if $s(\alpha)\neq \mathtt{i}$,}\\
F_\mathtt{i}^-(M)_{\beta^*}&:=N_\beta&\text{if $s(\beta)=\mathtt{i}$.})
\end{align*}
For a morphism $f\colon M\to \tilde{M}$, $F_\mathtt{i}^+(f)$ is defined by the universal property of the kernel (respectively $F_\mathtt{i}^-(f)$ is defined by the universal property of the cokernel).

It is proven in \cite[Section 5]{LY15} that if $\Lambda$ is a locally Frobenius dualisable species of algebras, then $F_\mathtt{i}$ is given by tilting at an analogue of an APR tilting module. In a forthcoming paper, joint with Chrysostomos Psaroudakis, we study conditions under which the module $I_\mathtt{i}$, giving rise to the reflection functors $(\Sigma_\mathtt{i}^+,\Sigma_\mathtt{i}^-)$ on the preprojective algebra is a tilting module. Furthermore, in joint work with Chrysostomos Psaroudakis and Nan Gao, we use reflection functors in the study of a generalisation of the submodule category as studied by Ringel and Schmidmeier, see e.g. \cite{RS08}.

\section{The separated pro-species}
\label{sec:separated}

In this section, we show that the well known results on the seperated quiver of a radical square zero algebra (see e.g. \cite[Section 10]{DR75}, \cite[Section X.2]{ARS95}) extend to our setting. 

\begin{defn}
Let $\Lambda\colon Q\to \Algpro$ be a pro-species of algebras. Define $\rep_{l.p.}^{epi}(\Lambda)$ to be the full subcategory of $\rep_{l.p.}$ such that $M_{\mathtt{i},\arrowin}$ is an epimorphism for all $\mathtt{i}\in Q_0$. 
\end{defn}

\begin{lem}\label{separateddecomposes}
Let $\Lambda\colon Q\to \Algpro$ be a locally selfinjective pro-species such that $\Lambda_\alpha\neq 0$ implies that $\Lambda_\beta\neq 0$ for each $\beta$ with $t(\beta)=s(\alpha)$. Then every object in $\rep_{l.p.}$ decomposes into a direct sum of objects in $\rep_{l.p.}^{epi}(\Lambda)$ and objects isomorphic to projective $\Lambda_\mathtt{i}$-modules regarded as $T(\Lambda)$-modules for some $\mathtt{i}\in Q_0$. 
\end{lem}

\begin{proof}
Without loss of generality we can assume that $Q$ is a bipartite quiver. Otherwise we can replace $Q$ with $\tilde{Q}$ where all the arrows with $\Lambda_\alpha=0$ are removed. Defining $\tilde{\Lambda}$ by $\tilde{\Lambda}_\mathtt{i}=\Lambda_{\mathtt{i}}$ for all $\mathtt{i}$ and $\tilde{\Lambda}_\alpha:=\Lambda_\alpha$ for $\alpha$ with $\Lambda_\alpha\neq 0$ one obtains an equivalence between $\rep_{l.p.}(\Lambda)$ and $\rep_{l.p.}(\tilde{\Lambda})$. For such a quiver let $I_0$ be the vertices such that there is $\alpha$ with $t(\alpha)\in I_0$ and $I_1:=Q_0\setminus I_0$. 

Let $M\in \modu_{l.p.}\Gamma$ be an arbitrary representation. Let $X$ be the $\Lambda$-representation with $X_\mathtt{i}:=\image M_{\mathtt{i},\arrowin}$ for $\mathtt{i}\in I_0$ and $X_\mathtt{i}:=M_{\mathtt{i}}$ otherwise. Define $X_\alpha$ to be the map induced by $M_\alpha$ for all $\alpha$. Since $\Lambda$ is locally selfinjective and $M$ is locally projective, the inclusion $X_\mathtt{i}\to M_\mathtt{i}$ splits. Thus, $X\in \rep_{l.p.}^{epi}(\Lambda)$. Furthermore $X$ is a subrepresentation of $M$ and we can form the quotient $N:=M/X$ which is easily seen to satisfy $N_\alpha=0$ for all $\alpha$. Since $N$ is also locally projective, it follows that $N$ is isomorphic to a direct summand of a direct sum of $\Lambda_\mathtt{i}$ for some $\mathtt{i}\in Q_0$. 
\end{proof}

\begin{rmk}
Without the assumption of $\Lambda$ being locally selfinjective, the result is not true as one can see from the example of $Q=\mathbb{A}_2=1\stackrel{\alpha}{\to} 2$, $\Lambda_\mathtt{i}:=k\mathbb{A}_2$ for $\mathtt{i}\in \{1,2\}$ and $\lambda_\alpha:=\Lambda_\mathtt{1}$. Then the representation $M$ with $M_\mathtt{1}=P_2=S_2$ and $M_{\mathtt{2}}=P_1$ and $M_\alpha$ being the natural inclusion is indecomposable, but neither in $\rep_{l.p.}^{epi}(\Lambda)$ nor isomorphic to $\Lambda_{\mathtt{i}}$ for some $\mathtt{i}$.
\end{rmk}

The assumptions of Lemma \ref{separateddecomposes} are in particular satisfied for the separated pro-species:

\begin{defn}
Let $\Lambda\colon Q\to \Algpro$ be a pro-species of algebras with $Q$ not necessarily acyclic. The \emphbf{separated pro-\-spe\-cies} $\Lambda^s$ associated to $\Lambda$ is defined as follows. Let $Q^s$ be the quiver with vertex set being the disjoint union $Q_0 \coprod Q_0$ where we denote a vertex in the latter set by $\overline{\mathtt{i}}$ for $\mathtt{i}\in Q_0$, and arrows $\overline{\alpha}\colon \mathtt{i}\to \overline{\mathtt{j}}$ for each arrow $\alpha\colon \mathtt{i}\to \mathtt{j}$ in $Q$. Let $\Lambda^s\colon Q^s\to \Algpro$ be the pro-species with $\Lambda^s_\mathtt{i}:=\Lambda^s_{\overline{\mathtt{i}}}:=\Lambda_\mathtt{i}$ and $\Lambda^s_{\overline{\alpha}}:=\Lambda_\alpha$. 
\end{defn}

The following proposition is essentially proven as in the classical case (see e.g. \cite[Lemma X.2.1]{ARS95}) replacing $\rad{T(\Lambda)}$ by $T(\Lambda)_{\geq 1}$. We provide it for convenience of the reader.

\begin{prop}
Let $\Lambda\colon Q\to \Algpro$ be a locally selfinjective pro-species of algebras with $Q$ not necessarily acyclic. Let $\Lambda^s$ be the separated pro-species associated to $\Lambda$. Let $T(\Lambda)$ be equipped with the tensor grading, i.e. $\Lambda_{\mathtt{i}}$ is in degree $0$ and $\Lambda_\alpha$ is in degree $1$ for $\alpha\in Q_1$. Let $T(\Lambda)_{\geq 2}$ be the part of degrees greater or equal than $2$. Let $\Gamma:=T(\Lambda)/T(\Lambda)_{\geq 2}$.   Then the functor 
\[F\colon \modu_{l.p.}\Gamma:=\{X\in \modu \Gamma\, |\, X\in \rep_{l.p.} \Lambda\}\to \rep_{l.p.}^{epi}\Lambda^s,\]
given on objects by $F(M)_{\overline{\mathtt{i}}}=\image M_{\mathtt{i},\arrowin}$ and $F(M)_\mathtt{i}=M_\mathtt{i}/\image M_{\mathtt{i},\arrowin}$ and $F(M)_\alpha$ being induced by $M_\alpha$, is full, dense and a representation embedding, i.e. it preserves indecomposability and reflects isomorphisms. 
\end{prop}

\begin{proof}
Let $\mathfrak{r}$ be the $\Gamma$-module $T(\Lambda)_{\geq 1}/T(\Lambda)_{\geq 2}$. Then, there is a short exact sequence of $\Gamma$-modules $0\to \mathfrak{r}\to \Gamma\to \Gamma/\mathfrak{r}\to 0$. Tensoring this with a projective $\Gamma$-module, one obtains a short exact sequence $0\to \mathfrak{r}\otimes_\Gamma P\to P\to P/\mathfrak{r}P\to 0$. Let $E$ be a $\Gamma$-module which is locally projective when considered as a $\Lambda$-module. Note that each such module is in fact gradable. Thus its $\Gamma$-projective cover comes from a graded map $f\colon P\to E$. Thus, there is the following commutative diagram of gradable maps with exact rows and columns
\[
\begin{tikzcd}
0\arrow{r} &\mathfrak{r}P\arrow{d}{s}\arrow{r}&P\arrow{d}\arrow{r} &P/\mathfrak{r}P\arrow{r}\arrow{d}{f_1}&0\\
0\arrow{r}&\mathfrak{r}E\arrow{d}\arrow{r}&E\arrow{r}\arrow{d}&E/\mathfrak{r}E\arrow{r}&0\\
&0&0
\end{tikzcd}
\] 
Note that $f_1$ is an isomorphism since $E$ is locally projective, which results in the left square being a pullback square as well as a pushout square.

Note that $T(\Lambda^s)$ is isomorphic to the triangular matrix ring \[\begin{pmatrix}T(\Lambda)/T(\Lambda)_{\geq 1}&0\\\mathfrak{r}&T(\Lambda)/T(\Lambda)_{\geq 1}\end{pmatrix}.\]
As such, $T(\Lambda^s)$-modules are given by triples $(M,N,f)$ where $M$ and $N$ are $T(\Lambda)/T(\Lambda)_{\geq 1}$-modules and $f\colon T(\Lambda)_{\geq 1}\otimes_{T(\Lambda)_{\geq 1}} M\to N$ is a $T(\Lambda)/T(\Lambda)_{\geq 1}$-linear map. Identifying $\mathfrak{r}P$ with $\mathfrak{r}\otimes_{T(\Lambda)/T(\Lambda)_{\geq 1}} P/\mathfrak{r}P$ we obtain a commutative diagram showing that $(P/\mathfrak{r}P, \mathfrak{r}E,s)$ is isomorphic to $(E/\mathfrak{r}E,\mathfrak{r}E,h)$ where $h$ is the map induced by the multiplication map.
\[
\begin{tikzcd}
\mathfrak{r}\otimes_{T(\Lambda)/T(\Lambda)_{\geq 1}} P/\mathfrak{r}P\arrow{d}{1\otimes f_1}\arrow{r}{s}& \mathfrak{r}E\arrow{d}{\id_{\mathfrak{r}E}}\\
\mathfrak{r}\otimes_{T(\Lambda)/T(\Lambda)_{\geq 1}} E/\mathfrak{r}E\arrow{r}{h}&\mathfrak{r}E
\end{tikzcd}
\] 
We first prove that the functor $F$ is full. For this let $E$ and $E'$ be in $\modu_{l.p.} \Gamma$ and write $F(E)=(M,N,t)$ and $F(E')=(M',N',t')$. Consider a morphism $(u,v)\colon (M,N,t)\to (M',N',t')$. Let $f\colon P\to E$ and $f'\colon P'\to E'$ be projective covers. According to the foregoing remarks there are $T(\Lambda^s)$-isomorphisms $(M,N,t)\cong (P/\mathfrak{r}P,N,s)$ and $(M',N',f')\cong (P'/\mathfrak{r}P',N',s')$. It thus suffices to prove that each morphism $(u',v)\colon (P/\mathfrak{r}P,N,s)\to (P'/\mathfrak{r}P',N',s')$ comes from a morphism $w\colon E\to E'$. For such a morphism $(u',v)$ as noted before there is a pushout square
\[
\begin{tikzcd}
\mathfrak{r}\otimes_{T(\Lambda)/T(\Lambda)_{\geq 1}} P/\mathfrak{r}P\arrow{d}{s}\arrow{r} &P\arrow{d}{f}\\
N\arrow{r}&E
\end{tikzcd}
\]
Let $g\colon P\to P'$ be a lift of $u'\colon P/\mathfrak{r}P\to P'/\mathfrak{r}P'$. There is a commutative diagram
\[
\begin{tikzcd}
\mathfrak{r}P\arrow{rd}{\mathfrak{r}u'}\arrow{dd}{s}\arrow{r}&P\arrow{rd}{g}\\
&\mathfrak{r}P'\arrow{dd}{s'}\arrow{r}&P'\arrow{dd}{f'}\\
N\arrow{rd}{v}\\
&N'\arrow{r}&E'
\end{tikzcd}
\]
where the left square comes from the fact that $(u',v)$ is a homomorphism of $T(\Lambda^s)$-modules, the lower right corner is the pushout square for $(P'/\mathfrak{r}P',N',s')$, and the upper square commutes as $g$ is a lift of $u'$. Combining the latter two diagrams one gets a map $w\colon E\to E'$ by the universal property of the pushout.
That $F(w)=(u',v)$ now follows from the commutativity of the following two cubes:
\[
\begin{tikzcd}
\mathfrak{r}P\arrow{rd}{g'}\arrow{dd}{s}\arrow{rr}{i}&&P\arrow{dd}[near end]{f}\arrow{rd}{g}\\
&\mathfrak{r}P'\arrow[crossing over]{rr}[near start]{i'}&&P'\arrow{dd}{f'}\\
\mathfrak{r}E\arrow{rd}{v}\arrow[]{rr}{h}&&E\arrow{rd}{w}\\
&\mathfrak{r}E'\arrow[from=uu, crossing over]{}[near end]{s'}\arrow{rr}{h'}&&E'
\end{tikzcd}
\]
where the lower square commutes as $whs=wfi=f'gi=f'i'g'=h's'g'=h'vs$ and $s$ is an epimorphism, and
\[
\begin{tikzcd}
P\arrow{rd}{g}\arrow{dd}{f}\arrow{rr}{\pi}&&P/\mathfrak{r}P\arrow{dd}[near end]{f_1}\arrow{rd}{u'}\\
&P'\arrow[crossing over]{rr}[near start]{\pi'}&&P'/\mathfrak{r}P'\arrow{dd}{f_1'}\\
E\arrow{rd}{w}\arrow{rr}[near end]{p}&&E/\mathfrak{r}E\arrow{rd}{u}\\
&E'\arrow[from=uu, crossing over]\arrow{rr}{p'}&&E'/\mathfrak{r}E'
\end{tikzcd}
\]
where the lower square commutes as $p'wf=p'f'g=f_1'\pi'g=f_1'u'\pi=uf_1\pi=upf$ and $f$ is an epimorphism.

We claim that $F$ preserves indecomposability. For this, first note that $F(f)=0$ if and only if $f(E)\subseteq \mathfrak{r}E'$. It follows that $\End_{T(\Lambda^s)}(F(E))\cong \End_{\Gamma}(E)/\Hom_{\Gamma}(E,\mathfrak{r}E)$. Since $\mathfrak{r}^2=0$, it follows that $\Hom_{\Gamma}(E,\mathfrak{r}E)\subseteq \rad{\End_\Gamma(E)}$. In particular, $\End_{T(\Lambda^s)}(F(E))$ is local if and only if $\End_{\Gamma}(E)$ is local. Hence, $F$ preserves indecomposability.

To show that $F$ reflects isomorphisms let $E$ and $E'$ be such that $F(E)\cong F(E')$. Since $F$ is full, there exist $f\colon E\to E'$ and $g\colon E'\to E$ with $F(fg)=\id$ and $F(gf)=\id$. But then, there exists $r\in  \rad{\End_\Gamma(E)}, r'\in \rad{\End_\Gamma(E')}$ with $gf+r=1$ and $fg+r'=1$. Since $r,r'$ are nilpotent, it follows that $gf$ and $fg$, and thus $f$ and $g$ are isomorphisms.

Finally, to prove that $F$ is dense let $(M,N,t)\in \rep_{l.p.}^{epi}\Lambda^s$. Let $f\colon P\to M$ be a projective cover. Taking the pushout of $\mathfrak{r}P\to P$ along the surjective $t\colon \mathfrak{r}P\to N$ gives a $\Gamma$-module $E$. We claim that $F(E)\cong (M,N,t)$. This follows from the pushout diagram as $P\to E$ being an epimorphism gives that the induced morphism $i(N)\to E$ has image $\mathfrak{r}E$.
\end{proof}

\begin{thm}\label{stableequivalence}
Let $\Lambda$, $\Gamma$, $\Lambda^s$, and $F$ be as above. Then $F$ induces an equivalence of the corresponding stable categories $\underline{\modu}_{l.p.}\Gamma\to \underline{\rep}_{l.p.}\Lambda^s$.
\end{thm}

\begin{proof}
As $T(\Lambda^s)$ can be written as a triangular matrix ring \[\begin{pmatrix}T(\Lambda)/T(\Lambda)_{\geq 1}&0\\T(\Lambda)_{\geq 1}/T(\Lambda)_{\geq 2}&T(\Lambda)/T(\Lambda)_{\geq 1}\end{pmatrix},\] it follows that $F$ preserves and reflects projectivity, see e.g. \cite[Lemma X.2.2]{ARS95}.

It follows that $F$ restricts to a functor $F\colon \modu_{l.p.}^{np} \Gamma\to \rep_{l.p.}^{np} \Lambda^s$. Since the $\Lambda^s$-re\-pre\-sen\-ta\-tions not in $\rep^{epi}_{l.p.}\Lambda^s$ are projective, see the proof of Lemma \ref{separateddecomposes}, it follows that this restriction is dense. Since, $F$ sends projectives to projectives, it follows that $F$ induces a dense functor $F'\colon \underline{\modu}_{l.p.} \Gamma\to \underline{\rep}_{l.p.} \Lambda^s$. By the foregoing lemma, it is as well full.

According to Lemma \ref{tensoralgebrahereditary}, $\rep_{l.p.}^{np}\Lambda^s\cong \underline{\rep}_{l.p.}\Lambda^s$. Suppose that $F'(f)=0$ for some morphism $f\colon M\to N$. Then, according to the proof of the previous proposition, it is in $\Hom_\Gamma(M,T(\Lambda)_{\geq 1}N)$. Thus there is a factorisation $M\to M/T(\Lambda)_{\geq 1} M\to T(\Lambda)_{\geq 1} N\to N$ of $f$. Let $g\colon P\to N$ be a $\Gamma$-projective cover of $N$. Then, $g$ induces an epimorphism $T(\Lambda)_{\geq 1} P\to T(\Lambda)_{\geq 1} N$. Noting that $\Lambda$ is locally selfinjective and the locally projective modules $T(\Lambda)_{\geq 1} P, T(\Lambda)_{\geq 1}, M/T(\Lambda)_{\geq 1} M$ are all modules for the selfinjective algebra $\prod \Lambda_{\mathtt{i}}$, it follows that the map $M/T(\Lambda)_{\geq 1}\to T(\Lambda)_{\geq 1}N$ factors through $T(\Lambda)_{\geq 1}P$. It follows that also $f$ factors through $P$.
\end{proof}

\section{Quivers and relations for pro-species}
\label{sec:quivers}

This final section of the article provides the bridge to classical representation theory of algebras by determining Gabriel quiver and relations for the tensor algebra as well as the preprojective algebra of a pro-species. This generalises a result by Ib{\'a}{\~n}ez Cobos, Navarro and L{\'o}pez Pe{\~n}a where the case of pro-species where each $\Lambda_\alpha$ is a free bimodule is considered under the name of a ``generalised path algebra". A similar result holds if the $\Lambda_\mathtt{i}$ on the vertices are given by classical species. We leave the obvious generalisation to the reader. 

\begin{prop}[cf. {\cite[Theorem 3.3]{INL08}}]\label{quivertensoralgebra}
Let $\Lambda\colon Q\to \Algpro$ be a pro-species of algebras. Suppose $\Lambda_\mathtt{i}\cong \mathbbm{k}\tilde{Q}_\mathtt{i}/R_\mathtt{i}$ is given by a quiver $\tilde{Q}_\mathtt{i}$ and relations $R_\mathtt{i}$. Let $\pi_\alpha\colon \tilde{P}_\alpha\to \Lambda_\alpha$ be a projective cover of $\Lambda_\alpha$ as a $\Lambda_{t(\alpha)}$-$\Lambda_{s(\alpha)}$-module. Denote its kernel by $R_\alpha$. Then, $T(\Lambda)\cong \mathbbm{k}\tilde{Q}/R$ is a description by a quiver with relations, where 
\[\tilde{Q}=\bigcup_{\mathtt{i}\in Q_0} \tilde{Q}_\mathtt{i}\cup \{j\to i\, |\, \Lambda_{t(\alpha)} e_j\otimes_\mathbbm{k} e_i\Lambda_{s(\alpha)} \text{ is a direct summand of } \tilde{P}_\alpha \text{ for some }\alpha\}\]
and $R=\langle R_\mathtt{i}, R_\alpha\rangle$.
\end{prop}

\begin{proof}
Firstly, we have to determine the Jacobson radical of $T(\Lambda)$. It is easy to see that the ideal $J$ spanned by the $\tilde{Q}_{\mathtt{i},1}$, the arrows of the quiver $Q_\mathtt{i}$ of $\Lambda_\mathtt{i}$, and the $\Lambda_\alpha$ is nilpotent as the quiver $Q$ is acyclic. Noting that $T(\Lambda)/J\cong \prod \Lambda_\mathtt{i}/J_\mathtt{i}$, where $J_\mathtt{i}$ is the Jacobson radical of $\Lambda_\mathtt{i}$ the first claim follows.

Secondly, to determine $J/J^2$, note that the arrows in $\tilde{Q}_\mathtt{i}$ correspond to a basis of $J_\mathtt{i}/J_\mathtt{i}^2$. Furthermore, for the $\Lambda_\alpha$ one notes that elements of $\Kopf(\Lambda_\alpha)$ do not belong to $J^2$. As the elements of $\Kopf(\Lambda_\alpha)$ correspond to the direct summands of the form $\Lambda_{t(\alpha)}e_j\otimes e_i\Lambda_{s(\alpha)}$, the description of the quiver follows.

Thirdly, it is clear that $T(\Lambda)\cong \mathbbm{k}\tilde{Q}/R$. What is left to prove is that $R$ is admissible. It is obvious that $(\mathbbm{k}\tilde{Q}^+)^m\subseteq R$. For the inclusion $R\subseteq (\mathbbm{k}\tilde{Q}^+)^2$ note that this follows from the corresponding fact for the $R_\mathtt{i}$ and the fact that we chose a projective cover of $\Lambda_\alpha$ and hence $R_\alpha\subseteq \rad{\tilde{P}_\alpha}$ whose summands are of the form $\rad{\Lambda_{t(\alpha)}}e_j\otimes_\mathbbm{k} e_i \Lambda_{s(\alpha)}+\Lambda_{t(\alpha)}e_j\otimes_\mathbbm{k} e_i\rad{\Lambda_{s(\alpha)}}$. Since $e_j\otimes_\mathbbm{k} e_i\in \rad{T(\Lambda)}$, the claim follows.
\end{proof}

The proof does not use the fact that the $\Lambda_\alpha$ are projective from both sides. We included this assumption because we use it everywhere else in the paper. It is not known to the author whether this property was already established for triangular matrix rings.

\begin{ex}\label{glsexample}
\begin{enumerate}[(a)]
\item\label{gls:i} Let $Q=1\stackrel{\alpha}{\to} 2$ and let $\Lambda$ be given by $\Lambda_1=\Lambda_2=\mathbbm{k}(1\stackrel{\beta}{\to} 2)$ and let $\Lambda_\alpha=\Lambda_1$. Then a projective cover of $\Lambda_\alpha$ is given by $\Lambda_2e_1\otimes e_1\Lambda_1\oplus \Lambda_2e_2\otimes e_2\Lambda_1$, the kernel of $\pi_\alpha$ is generated by $\beta\otimes e_1-e_2\otimes \beta$. We get the well known fact that this triangular matrix ring is given by the commutative square.
\item\label{gls:ii} Let $Q$ be a quiver. Let $c_\mathtt{i}, f_{\mathtt{i}\mathtt{j}}, f_{\mathtt{j}\mathtt{i}}, g_{\mathtt{i}\mathtt{j}}\in \mathbb{N}$ for $\mathtt{i},\mathtt{j}\in Q_0$. Let $Q\to \Algpro$ be the pro-species of algebras given by \[\Lambda_\mathtt{i}:=\mathbbm{k}[x_\mathtt{i}]/(x_\mathtt{i}^{c_\mathtt{i}})\] for $\mathtt{i}\in Q_0$ and \[\Lambda_{\alpha}:=\bigoplus_{g_{\mathtt{i}\mathtt{j}}} \left(\mathbbm{k}[x_\mathtt{i}]/(x_\mathtt{i}^{c_\mathtt{i}})\otimes_\mathbbm{k} \mathbbm{k}[x_\mathtt{j}]/(x_\mathtt{j}^{c_\mathtt{j}})\right)/\left(x_\mathtt{i}^{f_{\mathtt{j}\mathtt{i}}}\otimes_\mathbbm{k} 1-1\otimes_\mathbbm{k} x_\mathtt{j}^{f_{\mathtt{i}\mathtt{j}}}\right)\]
for $\alpha\in Q_1$. Then $T(\Lambda)$ is isomorphic to the algebra $H(C,D,\Omega)$ as defined in \cite{GLS16} where (H1) are the relations corresponding to $R_\mathtt{i}$ and (H2) correspond to the relations given by $R_\alpha$. This explains the relations (H1) and (H2) of \cite{GLS16} which at first sight might seem unnatural.
\end{enumerate}
\end{ex}

Similarly, one obtains the following statement for the preprojective algebra. Note, that this algebra is not necessarily finite dimensional, thus we do not use the expression ``quiver with relations'' here.

\begin{prop}\label{quiverpreprojectivealgebra}
Let $\Lambda\colon Q\to \Algpro$ be a dualisable pro-species of algebras. Suppose $\Lambda_\mathtt{i}\cong \mathbbm{k}\tilde{Q}_\mathtt{i}/R_\mathtt{i}$ is given by a quiver $\tilde{Q}_\mathtt{i}$ and relations $R_\mathtt{i}$. For each $\alpha\in \overline{Q}$ let $\pi_\alpha\colon \tilde{P}_\alpha\to \Lambda_\alpha$ be a projective cover of $\Lambda_\alpha$ as a $\Lambda_{t(\alpha)}$-$\Lambda_{s(\alpha)}$-module. Denote its kernel by $R_\alpha$. Then, $\Pi(\Lambda)\cong \mathbbm{k}\tilde{Q}/R$, where 
\[\tilde{Q}=\bigcup_{\mathtt{i}\in Q_0} \tilde{Q}_\mathtt{i}\cup \{j\to i\, |\, \Lambda_{t(\alpha)} e_j\otimes_\mathbbm{k} e_i\Lambda_{s(\alpha)} \text{ is a direct summand of } \tilde{P}_\alpha \text{ for some }\alpha\}\]
and $R=\langle R_\mathtt{i}, R_\alpha, c\rangle$ where $c$ is as in Definition \ref{preprojectivealgebra}.
\end{prop}

\begin{ex}
Let $Q$ be a quiver. Let $\Lambda\colon Q\to \Algpro$ be defined as in Example \ref{glsexample} \eqref{gls:ii}, then $\Pi(\Lambda)$ is isomorphic to the algebra $\Pi(C,D,\Omega)$ from \cite{GLS16} where (P1) corresponds to the $R_\mathtt{i}$, (P2) corresponds to the $R_\alpha$ and (P3) corresponds to the relations in $c$.
\end{ex}

\bibliographystyle{alpha}
\bibliography{publication}

\end{document}